\documentclass[11pt,a4paper,reqno,dvipsnames]{amsart}
\allowdisplaybreaks[4]

\usepackage{amsfonts}
\usepackage{amsmath}
\usepackage{amssymb}
\usepackage{amsthm}
\usepackage{amsxtra}
\usepackage{bbm}
\usepackage{braket}
\usepackage{comment}
\usepackage{extarrows}
\usepackage{graphicx}
\usepackage{latexsym}
\usepackage{mathrsfs}
\usepackage{yhmath}
\usepackage{mismath}
\usepackage{mathtools}

\usepackage{float}
\usepackage{booktabs}
\usepackage{verbatim}
\usepackage{xcolor}
\usepackage[normalem]{ulem}
\usepackage{caption,subcaption}

\usepackage{hyperref} 
\hypersetup{breaklinks, raiselinks}
\hypersetup{colorlinks, citecolor=blue, linkcolor=blue, urlcolor=blue}

\usepackage{bm}

\usepackage{tikz-cd}
\usepackage{tikz}
\usetikzlibrary{fit}

\usepackage[initials,lite]{amsrefs} 
\AtBeginDocument{%
    \def\MR#1{}
}

\usepackage{cleveref}
\Crefname{Lemma}{Lemma}{Lemmas}
\Crefname{Theorem}{Theorem}{Theorems}
\makeatletter
\def\cref@thmoptarg[#1]#2#3#4{%
    \ifhmode\unskip\unskip\par\fi%
    \normalfont%
    \trivlist%
    \let\thmheadnl\relax%
    \let\thm@swap\@gobble%
    \thm@notefont{\fontseries\mddefault\upshape}%
    \thm@headpunct{.}
    \thm@headsep 5\p@ plus\p@ minus\p@\relax%
    \thm@space@setup%
    #2
    \@topsep \thm@preskip               
    \@topsepadd \thm@postskip           
    \def\@tempa{#3}\ifx\@empty\@tempa%
        \def\@tempa{\@oparg{\@begintheorem{#4}{}}[]}%
    \else%
        \refstepcounter[#1]{#3}
        \@namedef{cref@#3@alias}{#1}
        \def\@tempa{\@oparg{\@begintheorem{#4}{\csname the#3\endcsname}}[]}%
    \fi%
\@tempa}%
\makeatother

\usepackage{fullpage}

\theoremstyle{plain}
\newtheorem{Theorem}{Theorem}[section]
\newtheorem{Lemma}[Theorem]{Lemma}
\newtheorem{Corollary}[Theorem]{Corollary}
\newtheorem{Proposition}[Theorem]{Proposition}

\theoremstyle{definition}

\newtheorem{Assumptions and Discussion}[Theorem]{Assumptions and Discussion}

\newtheorem{Example}[Theorem]{Example}
\newtheorem{Definition}[Theorem]{Definition}

\newtheorem{Remark}[Theorem]{Remark}

\newtheorem{Construction}[Theorem]{Construction}

\newtheorem{Observations}[Theorem]{Observations}

\theoremstyle{remark}

\newtheorem{Setting}[Theorem]{Setting}

\newtheorem*{acknowledgment*}{Acknowledgment}

\usepackage[shortlabels]{enumitem}
\SetEnumerateShortLabel{a}{\textup{(\alph*)}}
\SetEnumerateShortLabel{A}{\textup{(\Alph*)}}
\SetEnumerateShortLabel{1}{\textup{(\arabic*)}}
\SetEnumerateShortLabel{i}{\textup{(\roman*)}}
\SetEnumerateShortLabel{I}{\textup{(\Roman*)}}

\def\ceil#1{\left\lceil #1 \right\rceil}

\def\cone{\operatorname{cone}}

\def\deg{\operatorname{deg}}

\def\floor#1{\left\lfloor #1 \right\rfloor}

\def\ini{\operatorname{in}} 

\def\KK{{\mathbb K}}
\def\lcm{\operatorname{lcm}}

\def\reg{\operatorname{reg}}

\def\v{\operatorname{v}} 

\newcommand\bfm{\mathbf{m}}

\newcommand\bfP{\mathbf{P}}

\newcommand\calC{\mathcal{C}}

\newcommand\calJ{\mathcal{J}}

\newcommand\calP{\mathcal{P}}

\newcommand\frakP{\mathfrak{P}}
\newcommand\frakp{\mathfrak{p}}

\newcommand{\Ass}{\operatorname{Ass}}

\begin{document}

\title{The $\mathrm{v}$-number of generalized binomial edge ideals of some graphs}

\author{Yi-Huang Shen}
\address{School of Mathematical Sciences, University of Science and Technology of China, Hefei, Anhui, 230026, P. R.~China}
\email{yhshen@ustc.edu.cn}

\author{Guangjun Zhu$^{\ast}$}
\address{School of Mathematical Sciences, Soochow University, Suzhou, Jiangsu, 215006, P. R.~China}
\email{zhuguangjun@suda.edu.cn}

\thanks{$^{\ast}$ Corresponding author}


\begin{abstract}
    Let $G$ be a finite connected simple graph, and let $\mathcal{J}_{K_m,G}$ denote its generalized binomial edge ideal. By investigating the colon ideals of $\mathcal{J}_{K_m,G}$, we derive a formula for the local $\mathrm{v}$-number of $\mathcal{J}_{K_m,G}$ with respect to the empty cut set. Furthermore, we classify graphs for which this generalized binomial edge ideal has $\mathrm{v}$-numbers $1$ or $2$. When $G$ is a connected closed graph, we compute the local $\mathrm{v}$-number of $\mathcal{J}_{K_2,G}$ by generalizing the work of Dey et al. Additionally, under the condition that $G$ is Cohen--Macaulay, we derive formulas for the $\mathrm{v}$-number of $\mathcal{J}_{K_m,G}$ and $\mathcal{J}_{K_2,G}^k$, and show that the $\mathrm{v}$-number of $\mathcal{J}_{K_2,G}^k$ is a linear function of $k$.
\end{abstract}

\maketitle

\section{Introduction}

Let $R=\KK[x_1,\dots, x_n]=\bigoplus\limits_{d\ge 0} R_d$ be the standard graded polynomial ring in $n$ variables over a field $\KK$.
If $I\subset R$ is a graded ideal, then we denote the set of all associated primes of $I$ by $\Ass(I)$. For any $\frak{p}\in \Ass(I)$,
the \emph{local $\v$-number} of $I$ at $\frak{p}$ 
is defined to be
\[
    \v_{\frakp}(I)\coloneqq \min\{d\ge 0 \mid \text{\ there exists an\ } f\in R_d \text{\ such that\ } (I : f)=\frakp\}.
\]
The \emph{$\v$-number} 
of $I$, denoted by $\v(I)$, is defined as
\[
    \v(I)\coloneqq \min\{\v_{\frakp}(I)\mid \frakp\in \Ass(I)\}.
\]

The $\v$-number, named after Wolmer Vasconcelos, was introduced in \cite{MR4011111} to study the asymptotic behavior of the minimum distance of projective Reed-Muller type codes. This invariant has since 
been the subject of extensive study from algebraic and combinatorial perspectives. 

A significant body of work has focused on the $\v$-number of monomial ideals (see \cite{arXiv:2306.14243,MR4950319,MR4942694,MR4491066}).
Squarefree monomial ideals have an intrinsic combinatorial nature and their algebraic properties often carry information on their underlying combinatorial structures, and vice versa. In \cite{MR4139109} the authors consider these ideals as edge ideals of clutters and find a combinatorial expression for their $\v$-numbers. This combinatorial description has been
exploited to classify $W_2$ graphs (see \cite[Theorem 4.2]{MR4139109}) and to study the combinatorial structure of graphs whose edge ideals have a Cohen–Macaulay second symbolic power.
In \cite{MR4756096}, Saha studied the relation between the $\v$-number and the regularity of cover ideals of graphs. He showed that $\v(J(G))\le \reg(J(G))+1$, where $J(G)$ is the cover ideal of a graph
$G$, and the equality holds if $G$ is a complete multipartite graph.
In \cite{MR4942694}, the authors express the $\v$-number of the Stanley–Reisner ideal of a simplicial complex in terms of its Alexander dual complex and prove that the $\v$-number of 
a cover ideal is just two less than the initial degree of its syzygy module.

Let \(I \subset R\) be a graded ideal. Independently, Kodiyalam
\cite{MR1621961} and Cutkosky, Herzog, and Trung \cite{MR1711319} proved that
\(\reg(I^k)\) is a linear function in \(k\) for all sufficiently large \(k\).
Motivated by these results, Conca \cite{MR4741232} established that the
function \(\v(I^k)\) is eventually linear in \(k\); more precisely, there exist
constants \(a\) and \(b\) such that \(\v(I^k) = ak + b\) for all \(k \gg 0\).
For the case where \(R\) is a polynomial ring over a field, Ficarra and Sgroi
independently obtained the same result (see \cite[Theorem
3.1]{arXiv:2306.14243}). They conjectured in 
\cite{MR4950310}
that if \(I\) has linear powers, then
\(\v(I^k) = \alpha(I)k - 1\) for all integers \(k \ge 1\), where \(\alpha(I)\)
stands for the initial degree of \(I\). 
They confirmed this conjecture for several interesting classes of graded ideal in  \cite{MR4950310,MR4932665} and subsequent papers.

In this paper, we focus on the $\v$-number of powers of generalized binomial edge ideals. First, we recall the definition of a generalized binomial edge ideal.

Let \(m\) and \(n\) be positive integers with \(m, n \ge 2\). Following standard convention, the notation \([m]\) denotes the set \(\{1, 2, \ldots, m\}\). Let \(S = \KK[\mathbf{X}] \coloneqq \KK[x_{i,j} : i \in [m], j \in [n]]\) be the polynomial ring over the field \(\KK\) in \(m \times n\) variables. 

In \cite{MR3290687}, Ene et al.~introduced the binomial edge ideal of a pair of graphs. Specifically, let $G_1$ and $G_2$ be simple graphs on vertex sets $[m]$ and $[n]$, respectively. Let $\varepsilon_1 = \{i, j\} \in E(G_1)$ and $\varepsilon_2 = \{k, l\} \in E(G_2)$ be edges with $i < j$ and $k < l$. One can then assign a $2$-minor $p_{(\varepsilon_1,\varepsilon_2)} = [i,j \mid k,l] \coloneqq x_{i,k}x_{j,l} - x_{i,l}x_{j,k}$ to the pair $(\varepsilon_1, \varepsilon_2)$. The \emph{binomial edge ideal of the pair $(G_1, G_2)$} is defined as
\[
    \calJ_{G_1,G_2} \coloneqq \left( p_{(\varepsilon_1,\varepsilon_2)} \mid \varepsilon_1 \in E(G_1), \varepsilon_2 \in E(G_2) \right)
\]
in $S$. This ideal generalizes the classical \emph{binomial edge ideals} introduced in \cite{MR2669070, MR2782571}, which are recovered when one of $G_1$ and $G_2$ is the complete graph $K_2$. 

Let $K_m$ be the complete graph on the vertex set \([m]\) and $G$ be a simple graph on the vertex set \([n]\); then $\calJ_{K_m,G}$ is the \emph{generalized binomial edge ideal} associated with $G$, first introduced by Rauh in \cite{MR3011436} for the study of conditional independence ideals.
Note that the ideal $\frakP_\emptyset\coloneqq (p_{(\varepsilon_1,\varepsilon_2)} \mid \varepsilon_1 \in E(K_m), \varepsilon_2 \in E(K_n))$ is always a minimal prime ideal of $\calJ_{K_m,G}$.

The (local) $\v$-numbers of binomial edge ideals have recently been discussed
in \cite{MR4834446,MR4869330,MR4792768,MR4984037}.  In \cite{MR4834446}, the
authors studied the properties and bounds of the $\v$-number of binomial edge
ideals. They characterized all connected graphs $G$ with $\v(\calJ_{K_2,G})=1$,
and showed that $\v_{\frakP_\emptyset}(\calJ_{K_2,G})=\textup{min-comp}(G)$,
where \(\textup{min-comp}(G)\) is the minimum completion number of the simple
graph $G$. In addition, when \(G\) is a connected non-complete graph,
Jaramillo-Velez and Seccia in \cite{MR4792768} proved that
$\v_{\frakP_\emptyset}(\calJ_{K_2,G})=\gamma_c(G)$, where \(\gamma_c(G)\) is
the connected domination number of $G$.  In \cite{MR4869330}, amongst many
beautiful results, Dey et al.~characterized all connected graphs $G$ with
$\v(\calJ_{K_2,G})= 2$, as well as computing the $\v$-number of the binomial
edge ideal of a Cohen--Macaulay closed graph. Building on the work of
\cites{MR4869330,arXiv:2507.02161}, Kumar et al.~in \cite{MR4984037} proved
that $\v(J_{C_n})=\lceil\frac{2n}{3}\rceil$ for all $n\ge 6$, where $C_n$ is a cycle with $n$ vertices. 
Nevertheless, to date, no literature has examined the $\v$-number of generalized binomial edge ideals.

The main focus of this paper is to generalize the work of Dey et al.~on the \(\v\)-number of the binomial edge ideal of a connected Cohen--Macaulay closed graph. 
This generalization is achieved in three non-trivial aspects. Specifically, we explicitly determine three key quantities: the local \(\v\)-number of the binomial edge ideal of a connected closed graph, the \(\v\)-number of the generalized binomial edge ideal of a connected Cohen--Macaulay closed graph, and the \(\v\)-number of powers of the binomial edge ideal of a connected Cohen--Macaulay closed graph. We address each of these objectives in Sections \ref{sec:nonCM_closed_graph} through \ref{sec:power}. 
In particular, the \(\v\)-number of powers of the generalized binomial edge ideal of a connected Cohen--Macaulay closed graph is a linear function of the power.

To achieve these goals, we will first review several necessary definitions and terms in the next section. Subsequently, in Section \ref{sec:gener}, we focus on the colon ideals of generalized binomial edge ideals. As a result, we derive a formula for 
$\v_{\frakP_\emptyset}(\calJ_{K_m,G})$
in terms of the minimum completion number and the connected domination number. Furthermore, we classify graphs for which the generalized binomial edge ideals have $\v$-numbers $1$ or \(2\). Influenced by the results in Section \ref{sec:gener}, it is natural to conjecture that the \(\v\)-number of the generalized binomial edge ideal \(\calJ_{K_m,G}\) is independent of the number \(m\). However, as our result in Section \ref{sec:closegraph} shows, this is not the case, and the computation of \(\v\)-numbers of generalized binomial edge ideals is undoubtedly more involved.

\section{Preliminary}

Let \(G\) be a simple graph with vertex set \(V(G)\) and edge set \(E(G)\).
In what follows, we assume that \(V(G) = [n] \coloneqq \{1, 2, \dots, n\}\) for some positive integer \(n\).
More generally, for non-negative integers \(n_1 \le n_2\), we write \([n_1,n_2] \coloneqq \{n_1, n_1+1, \dots, n_2\}\).

For any subset \(A\) of \(V(G)\), let \(G[A]\) denote the \emph{induced subgraph} of \(G\) on the vertex set \(A\); that is, for \(i, j \in A\), \(\{i, j\} \in E(G[A])\) if and only if \(\{i, j\} \in E(G)\). We denote the induced subgraph of \(G\) on \(V(G) \setminus A\) by \(G \setminus A\). For simplicity, if \(A = \{v\}\) is a singleton, we write this induced subgraph as \(G \setminus v\).

Complete graphs, path graphs, and cones are common simple graphs.

\begin{Definition} 
    \begin{enumerate}[a]
        \item A graph $G$ is called a \emph{complete graph} if there is an edge between every pair of its vertices. If $G$ has $n$ vertices, we often denote it by $K_n$.
        \item  A \emph{path graph} with $n$ vertices, denoted by $P_n$, is a graph whose vertex set can be ordered as $v_1,\ldots,v_n$ such that $E(P_n)=\{\{v_i,v_{i+1}\}\mid 1\le  i\le  n-1\}$. The \emph{length} of $P_n$ is the number of edges in $P_n$, which is $n-1$. 
        \item A graph $G$ is called a \emph{cone} if there is 
            subgraph $H$ and  a vertex $v\in V(G)\setminus V(H)$ such that $ V(G)=V(H)\cup\{v\}$ and $E(G)=E(H)\cup\{\{u,v\}\mid u\in V(H)\}$. In this case, we often denote it by $\cone(v, H)$.
    \end{enumerate}
\end{Definition}

As introduced earlier, given a simple graph $G$ with $n$ vertices and  an integer $m\ge 2$, we can consider the generalized binomial edge ideal $\calJ_{K_m,G}$ in $S=\KK[x_{1,1},\dots,x_{m,n}]$. Throughout this paper, we always adopt the lexicographic order induced by 
\begin{equation}
    x_{1,1}>x_{1,2}>\cdots >x_{1,n}>x_{2,1}>x_{2,2}>\cdots >x_{2,n}>\cdots >x_{m,1}>\cdots >x_{m,n}.
    \label{eqn:lex_order}
\end{equation}

It is well-known that $\calJ_{K_m,G}$ is a radical ideal. Its minimal prime ideals can be described explicitly as follows:
Let $c(G)$ denote the number of connected components of the graph $G$. A vertex $v\in V(G)$ is called a \emph{cut vertex} of $G$ if $c(G)< c(G\setminus v)$. For a subset $T\subseteq V(G)$, we abuse notation by letting $c(T)$ denote the number of connected components of $G\setminus T$. 
We say that a subset \(T\) of \(V(G)\) is a \emph{cut set} of \(G\) if, for every \(v\in T\), \(v\) is a cut vertex of the induced subgraph \(G\setminus (T\setminus \{v\})\).  We define \(\mathcal{C}(G)\) to be the 
collection of all cut sets of \(G\).  Note that \(\emptyset\in \mathcal{C}(G)\).

For each subset $T\subseteq [n]=V(G)$, we define the ideal 
\[
    \frakP_T(K_m, G)\coloneqq (x_{i,j}: (i,j)\in [m]\times
    T)+\calJ_{K_m,\widetilde{G_1}}+\cdots+\calJ_{K_m,\widetilde{G_{c(T)}}}
\]
in $S$, where $G_1,\ldots, G_{c(T)}$ are the connected components of $G\setminus T$, and $\widetilde{G_i}$ is the complete graph over the vertices of $G_i$ for $i=1,\dots,c(T)$. 
It is well-known that $\calJ_{K_m,G}$ is a radical ideal with $\Ass(\calJ_{K_m,G})=\{\frakP_T(K_m, G)\mid T\in \calC(G)\}$. 
For simplicity, the local $\v$-number of $\calJ_{K_m,G}$ with respect to $\frakP_T(K_m,G)$ will be denoted by $\v_T(\calJ_{K_m,G})$.

The main focus of this paper is to study the (local) $\v$-numbers associated to the generalized binomial edge ideal of a connected closed graph.
Closed graphs have many interesting characterizations and beautiful properties. For instance, we can choose the following as its definition:

\begin{Definition}
    [{\cite[Theorem 2.2]{MR2863365} and \cite[Theorem 1.3]{MR3290687}}]
    Let \(G\) be a simple graph on the vertex set \([n]\). Then the following conditions are equivalent:
    \begin{enumerate}[a]
        \item The minimal generators \(\{p_{(\varepsilon_1,\varepsilon_2)} \mid \varepsilon_1 \in E(K_m), \varepsilon_2 \in E(G)\}\) of \(\calJ_{K_m,G}\) form a quadratic Gröbner basis with respect to the lexicographic order defined in \eqref{eqn:lex_order};
        \item For all integers \(1\le i < j < k\le n\), if \(\{i, k\}\in E(G)\) then \(\{i,j\}\in E(G)\) and \(\{j, k\}\in E(G)\);
        \item All facets of the clique complex $\Delta(G)$ of $G$ are intervals of the form $[a, b] \subset [n]$.
    \end{enumerate}
    A graph \(G\) is \emph{closed} if there exists a vertex labeling for which one of the above equivalent conditions holds.
\end{Definition}

In this paper, we say that \(G\) is \emph{Cohen--Macaulay} if \(S/\calJ_{K_2,G}\) is Cohen--Macaulay.  By \cite[Theorem 3.1]{MR2863365}, a connected Cohen--Macaulay closed graph \(G\) with \(n\) vertices has a vertex labeling and integers \(1 = a_1 < a_2 < \cdots < a_t < a_{t+1} = n\) for which the maximal cliques \(F_1,\ldots,F_{t+1}\) of \(G\) satisfy \(F_i = [a_i,a_{i+1}]\) for all \(i=1,\dots,t\).

Our investigation of the $\v$-number of a connected closed graph depends heavily on the knowledge of the local $\v$-number with respect to $\frakP_{\emptyset}=\frakP_{\emptyset}(K_m,G)$. When $m=2$, there are two equivalent approaches for handling this.

For a connected graph $G$, the notion of connected domination number is widely used; see, for instance, \cite{MR4792768}. For the simplicity of the subsequent study, we introduce some variance of this notion. A \emph{reduced connected dominating set} is a subset $D$ of its vertices such that the induced subgraph $G[D]$ is connected, and every two vertices $u,v$ in $V(G)\setminus D$ can be connected with a path $u=u_0,u_1,\dots,u_s,u_{s+1}=v$ such that $u_1,\dots,u_s\in D$. The \emph{reduced connected domination number} of $G$, denoted by $\gamma_{c}^*(G)$, is the minimum cardinality of all such sets.  In this language, if $G$ is a complete graph, then the empty set $\emptyset$ is a reduced connected dominating set, and $\gamma_c^*(G)=0$. On the other hand, if $G$ is not complete, then $D$ is a reduced connected dominating set of $G$ if and only if $D$ is a connected dominating set of $G$. Whence, $\gamma_c^*(G)$ is precisely the connected domination number $\gamma_c(G)$ of $G$. One of the main results of \cite{MR4792768} shows that $\v_{\frakP_\emptyset}(\calJ_{K_2,G})$ is given by $\gamma_c^*(G)$.

Relatedly, for a vertex \(v\) of \(G\), its \emph{neighborhood} is defined as \(N_G(v) \coloneqq \{u \in V(G) \mid \{u,v\} \in E(G)\}\).
Furthermore, the \emph{completion graph} of \(G\) with respect to the vertex \(v\) is the graph \(G_v\), which is defined by
\[
V(G_v) = V(G) \quad \text{and} \quad E(G_v) = E(G) \cup \{\{u,w\} \mid u,w \in N_G(v), u \neq w\}.
\]

More generally, the \emph{completion graph} of \(G\) with respect to a subset \(V = \{v_1, \ldots, v_k\} \subseteq V(G)\) is the graph \(G_V\) defined iteratively as \(G_V = G_{v_1v_2\cdots v_k} \coloneqq (\ldots ((G_{v_1})_{v_2})\ldots)_{v_k}\). \cite[Proposition 3.2]{MR4834446} shows that the definition of \(G_V\) is well-defined. A subset \(W \subseteq V(G)\) is called a \emph{completion set} if \(G_W\) is a disjoint union of complete graphs. A \emph{minimal completion set} is a completion set \(W\) such that \(G_U\) is not a disjoint union of complete graphs for any proper subset \(U\) of \(W\). The \emph{minimum completion number}, denoted \(\textup{min-comp}(G)\), is the minimum cardinality of all minimal completion sets of \(G\). It is shown in \cite{MR4834446} that \(\v_{\mathfrak{P}_\emptyset}(\calJ_{K_2,G})\) is also equal to \(\textup{min-comp}(G)\).

\section{$\v$-number of generalized binomial edge ideals}
\label{sec:gener}

As is evident from the work in \cites{MR4834446}, colon ideals of elements associated to vertices are fundamental to the study of the \(\v\)-number of binomial edge ideals. Therefore, we start by studying those colon ideals for generalized binomial edge ideals.

\begin{Lemma}
    \label{thm:colon_x}
    Let \(G\) be a simple graph on the vertex set \([n]\). Then for any \(i \in [m]\) and \(j \in [n]\), we have \((\calJ_{K_m,G} : x_{i,j}) = \calJ_{K_m,G_j}\), where \(G_j\) is the completion graph of \(G\) with respect to the vertex \(j\).
\end{Lemma}
\begin{proof}
    Firstly, we show that $(\calJ_{K_m,G}: x_{i,j}) \supseteq \calJ_{K_m, G_j}$. Since
    \[
        \calJ_{K_m,G_j}=\calJ_{K_m,G}+(x_{k,p}x_{l,q}-x_{k,q}x_{l,p}\mid \text{$p,q\in N_G(j)$ and  $1\le k<l\le m$}),
    \]
    it suffices to show that $x_{i,j}(x_{k,p}x_{l,q}-x_{k,q}x_{l,p})\in \calJ_{K_m,G}$ for any $p,q\in N_G(j)$ and  $1\le k<l\le m$. But this can be checked straightforwardly. Observe that
    \begin{align*}
        x_{i,j}(x_{k,p}x_{l,q}-x_{k,q}x_{l,p})&= x_{k,p}(x_{i,j}x_{l,q}-x_{i,q}x_{l,j})+x_{i,q}(x_{k,p}x_{l,j}-x_{k,j}x_{l,p})\\
        & \qquad +x_{l,p}(x_{i,q}x_{k,j}-x_{i,j}x_{k,q}).
    \end{align*}
    Since $\{p,j\},\{q,j\}\in E(G)$, we have 
    \[
        x_{i,j}x_{l,q}-x_{i,q}x_{l,j}, x_{k,p}x_{l,j}-x_{k,j}x_{l,p}, x_{i,q}x_{k,j}-x_{i,j}x_{k,q}\in \calJ_{K_m,G}.
    \]
    Hence, $x_{i,j}(x_{k,p}x_{l,q}-x_{k,q}x_{l,p})\in \calJ_{K_m,G}$, as expected.

    It remains to show that $(\calJ_{K_m,G}: x_{i,j}) \subseteq \calJ_{K_m, G_j}$. For this purpose, let us take arbitrary $f\in (\calJ_{K_m,G}:x_{i,j})$. Whence,
    \[
        x_{i,j}f\in \calJ_{K_m,G}\subseteq \calJ_{K_m,G_j}\subseteq \frakP_T(K_m,G_j),
    \]
    for every $T\in \calC(G_j)$. By \cite[Lemma 4.5 (1)]{MR4423525}, 
    this implies that $T\in \calC(G)$ and $j\notin T$. Since \(x_{i,j}\notin \frakP_T(K_m,G_j)\), we have \(f\in \frakP_T(K_m,G_j)\) for every \(T\in \calC(G_j)\). In other words, \(f\in \calJ_{K_m,G_j}\). From this, we conclude that $(\calJ_{K_m,G}: x_{i,j}) \subseteq \calJ_{K_m, G_j}$.
\end{proof}

\begin{Corollary}
    \label{cor:colon}
    Let $G$ be a simple graph on the vertex set $[n]$. Suppose that $\varepsilon_1\in E(K_m)$ and $\varepsilon_2=\{k,l\}\in E(\widetilde{G})\setminus E(G)$, such that $P: k, k_1,\ldots, k_s, l$ is a path from \(k\) to \(l\) in \(G\). Then
    \[
        x_{i_1,k_1}x_{i_2,k_2}\cdots x_{i_s,k_s} p_{(\varepsilon_1,\varepsilon_2)} \in \calJ_{K_m,G},
    \]
    where $i_1,\ldots,i_s\in [m]$.
\end{Corollary}
\begin{proof}
    Let \(G'=(\cdots(G_{k_1})_{k_2}\cdots)_{k_s}\). It follows from \Cref{thm:colon_x} that
    \[
        (\calJ_{K_m,G}:x_{i_1,k_1}x_{i_2,k_2}\cdots x_{i_s,k_s})=\calJ_{K_m,G'}.
    \]
    Since \(\varepsilon_2\in E(G')\), the expected containment holds.
\end{proof}

\begin{Theorem}
    \label{cor:v_empty}
    Let \(G\) be a  simple graph on the set \([n]\) and \(m\ge 2\). Then
    \[
        \v_{\emptyset}(\calJ_{K_m,G})=
        \v_{\emptyset}(\calJ_{K_2,G})=\textup{min-comp}(G)=\gamma_c^{*}(G),
    \]
    where \(\textup{min-comp}(G)\) is the minimum completion number of $G$ and \(\gamma_c^*(G)\) is the reduced connected domination number of $G$.
\end{Theorem}
\begin{proof}
    It follows from \cite[Theorem 3.2]{MR4792768} that \(\v_{\emptyset}(\calJ_{K_2,G})=\gamma_c^{*}(G)\). At the same time, it follows from \cite[Theorem 3.6]{MR4834446} that \(\v_{\emptyset}(\calJ_{K_2,G})=\textup{min-comp}(G)\). Notice that the current \Cref{thm:colon_x} is a generalization of \cite[Proposition 3.1]{MR4834446} to the generalized binomial edge ideal case. One can also easily generalize \cite[Lemma 3.5]{MR4834446} to the generalized binomial edge ideal case using exactly the same argument. Therefore, we can do the same trick as in \cite[Theorem 3.6]{MR4834446} to get that \(\v_{\emptyset}(\calJ_{K_m,G})=\textup{min-comp}(G)\).
\end{proof}

\begin{Theorem}
    Let $G$ be a simple graph on the vertex set $[n]$. Then $\v(\calJ_{K_m,G})=1$ if and only if $G=\cone(v,H)$ is a cone graph for some non-complete graph $H$.
\end{Theorem}
\begin{proof}
    \begin{enumerate}[a]
        \item Suppose that \(\v(\calJ_{K_m,G})=1\). Then $G$ is not a complete graph and it follows from the definition that there is a homogeneous polynomial \(f\) of degree one such that \((\calJ_{K_m,G}:f)=P_{T_0}(K_m,G)\) for some \(T_0\in \calC(G)\). We claim that \(T_0=\emptyset\). Suppose for contradiction that this is not true. Then there exists \(x_{i,j}\in P_{T_0}(K_m,G)\) for some \(i\in [m]\) and \(j\in T_0\). It follows that \(x_{i,j}f\in \calJ_{K_m,G}\subseteq P_{\emptyset}(K_m,G)\). Since \(P_{\emptyset}(K_m,G)\) is a prime ideal, we have either \(x_{i,j}\) or \(f\) belonging to \(P_{\emptyset}(K_m,G)\). But \(\deg(x_{i,j})=\deg(f)=1\) while \(P_{\emptyset}(K_m,G)\) is generated by quadratic polynomials, this is impossible. Therefore, our claim holds and \(T_0=\emptyset\). Furthermore, for the given \(f\), we have \((\calJ_{K_m,G}:f)=P_{\emptyset}(K_m,G)=\calJ_{K_m,K_n}\).

            Since \(\calJ_{K_m,G}=\bigcap_{T\in\calC(G)}\frakP_T(K_m,G)\) is a minimal  prime (primary) decomposition of the  radical ideal \(\calJ_{K_m,G}\), \(P_{\emptyset}(K_m,G)=(\calJ_{K_m,G}:f)=\bigcap_{T\in\calC(G)}(\frakP_T(K_m,G):f)\), which implies that \(f\in \frakP_T(K_m,G)\) for every non-empty \(T\) in \(\calC(G)\). Since \(\deg(f)=1\), this implies that \(f\in (x_{i,j}\mid i\in [m],j\in T)\) for such \(T\). Consequently, \(f\in \bigcap_{\emptyset\ne T\in \calC(G)}\sum_{i\in [m],j\in T}\KK x_{i,j}\). Let \(x_{i,v}\) be a variable appearing in the expression of \(f\). This implies that \(v\in T\) for every non-empty \(T\) in \(\calC(G)\). It follows that \(G=\cone(v,H)\) for some non-complete graph \(H\); see also the proof of \cite[Theorem 3.20]{MR4834446}.

        \item Conversely, suppose that \(G=\cone(v,H)\) for some non-complete graph \(H\). It follows from \Cref{thm:colon_x} that \((\calJ_{K_m,G}:x_{i,v})=\calJ_{K_m,G_{v}}\).
            Since  \(G=\cone(v,H)\),  $G_{v}$ is complete and $\calJ_{K_{m},G_v}=P_{\emptyset}(K_m,G)$. 
            Therefore, \(\v(\calJ_{K_{m},G})\le 1\). Since \(G\) is not a complete graph, \(\calJ_{K_{m},G}\) is not a prime ideal and \(\v(\calJ_{K_{m},G})\ne 0\). Therefore, \(\v(\calJ_{K_{m},G})= 1\).
            \qedhere
    \end{enumerate} 
\end{proof}

Similar to the arguments in \cite[Theorem 4.1]{MR4869330}, we obtain the following theorem, whose proof we omit.

\begin{Theorem}
    Let $G$ be a connected graph. Then  \(\v(\calJ_{K_{m},G})=2\) if and only if $G$ is  not a cone graph and satisfies one of the following two conditions:
    \begin{enumerate}[1]
        \item  The reduced connected domination number of $G$,  $\gamma_c^{*}(G)=2$;
        \item   There exist vertices $u,v$ such  that $\{u,v\}\notin E(G)$, and  $N_G(u)\cap N_G(v)$ is a non-empty cut set of $G$ that disconnects $u$ and $v$. If $G_1, G_2\in  \calC(G\setminus (N_G(u)\cap N_G(v))$
            are components containing $u$ and $v$ respectively, then $G_1=\cone(u, G_1\setminus\{ u\})$ and 
            $G_2=\cone(v, G_2\setminus\{v\})$. All other connected components of $G\setminus (N_G(u)\cap N_G(v))$ are complete graphs.
    \end{enumerate} 
\end{Theorem}

So far, the results in this section are straightforward generalizations of the corresponding binomial edge ideal results to the setting of generalized binomial edge ideals. We nonetheless list a few more such results for subsequent application, whose proofs follow almost identically to the original arguments and are thus omitted for brevity.

\begin{Lemma}
    [see also {\cite[Corollary 3.12]{MR4834446}}]
    Let $G = G_1 \sqcup G_2$ be the disjoint union of graphs $G_1$ and $G_2$. Then $\v(\calJ_{K_m,G}) = \v(\calJ_{K_m,G_1}) + \v(\calJ_{K_m,G_2})$.
\end{Lemma}
Due to this result, we will focus on connected closed graphs in subsequent sections.

\begin{Definition}
    \label{def:G_f}
    Let $G$ be a simple graph on the vertex set $[n]$ and a non-edge $\varepsilon=\{k,l\}$. We define $G_\varepsilon^\dag$ to be the graph on $[n]$ with edge set
    \[
        E(G_\varepsilon^\dag)= E(G)\cup\{\{u, v\} \mid u, v\in N_G(k) \text{\ or\ } u,v \in  N_G(l)\}.
    \]
\end{Definition}

The notation in \Cref{def:G_f} differs from the standard usage in \cite{MR3169597},
since we wish to avoid confusion with the completion graph of \(G\) with respect to the set \(\{k,l\}\subset V(G)\).

\begin{Lemma}
    [see also {\cite[Theorem 3.7]{MR3169597}}]
    \label{thm:colon_edge}
    Let $G$ be a simple graph on the set $[n]$ and $\varepsilon_1\in E(K_m)$. Then for every non-edge $\varepsilon_2=\{k,l\}$ of \(G\), we have
    \begin{align*}
        (\calJ_{K_m,G}:p_{(\varepsilon_1,\varepsilon_2)})
        =\calJ_{K_m,G_{\varepsilon_2}^\dag}&+\big(x_{i_1,k_1}x_{i_2,k_2}\cdots x_{i_s,k_s}\mid  \text{$i_1,i_2,\ldots,i_s\in [m]$,}\\
        &\text{and there is  a path\ } P: k, k_1,\ldots, k_s, l \text{\ from $k$ to $l$ }\big).
    \end{align*}
\end{Lemma}

\begin{Lemma}
    [see also {\cite[Proposition 3.3]{MR4869330}}]
    \label{lem:change_initial}
    Suppose that $I$ is a graded ideal of $S$ and that  $f$ is a homogeneous form such that $(I:f)=\frakP$ for some $\frakP\in \Ass(I)$. Then, there exists a homogeneous form $g$ of the same degree such that $(I:g)=\frakP$ and $\ini(g)\notin \ini(I)$.
\end{Lemma}

 We conclude this section by emphasizing that generalized binomial edge ideals are more involved, with the study of their \(\v\)-numbers being no exception, as is evident from our analysis in Section~\ref{sec:closegraph}.

\section{Local \(\v\)-numbers of  binomial edge ideals of connected closed graphs} 
\label{sec:nonCM_closed_graph}

In \cite[Theorem 3.4]{MR4869330}, Dey et al.~established the local v-numbers of binomial edge ideals associated with connected Cohen--Macaulay closed graphs. In this section, we extend this result by computing the same invariants for binomial edge ideals without the Cohen--Macaulay assumption.

\begin{Setting}
    \label{set:closed_graph_non_CM}
    Let $G$ denote a connected closed graph. After labeling, we can assume that its clique complex $\Delta(G)=\langle F_1,\ldots,F_t\rangle$ has facets $F_i=[a_i,b_i]$ satisfying $1=a_1<\cdots < a_t<b_t=n$. For any $1\leq i\leq t-1$, let $W_i\coloneqq F_i\cap F_{i+1}$. We call the sets \(W_i\) the \emph{connected cut sets} of $G$.
\end{Setting}

In the following, we will always assume that \(G\) is a connected  closed graph, which satisfies the assumptions in \Cref{set:closed_graph_non_CM}.  We start by considering the local \(\v\)-number of \(\calJ_{K_m,G}\) with respect to the empty set \(\emptyset\) in \(\calC(G)\).
\begin{Construction}
    \label{constr:CDS}
    Let \(G\) be a connected closed graph.
    Set \(c_0=1=a_1\) and \(c_1=b_1\). Furthermore, for each \(c_i\), if \(c_i\ne n\), then we set \(c_{i+1}=\max\{\max(F_k)\mid c_i\in F_k\}\). Therefore, we have vertices \(c_0=1,c_1,\dots,c_{d},c_{d+1}=n\) for some  positive integer \(d\). These vertices form a path of minimal length connecting \(1\) to \(n\). 
\end{Construction}

\begin{Example}
    Let \(G\) be a connected closed graph with maximal cliques \([1,5]\), \([3,6]\), \([4,8]\), \([5,10]\) and \([7,12]\). Then the  vertex set of the  path constructed in \Cref{constr:CDS} is \(\{1,5,10,12\}\) with \(d=2\).
\end{Example}

\begin{Proposition}
    \label{lem:minimal_CDS}
    The number \(d\) in \Cref{constr:CDS} is the reduced connected domination number \(\gamma_c^{*}(G)\) of \(G\). In particular, it gives \(\v_{\emptyset}(\calJ_{K_m,G})\) for every \(m\ge 2\).
\end{Proposition}
\begin{proof}
    First, we will show that \(\{c_1,\dots,c_d\}\) is a reduced connected dominating set of \(G\). To see this, notice that \(c_i\) is connected to \(c_{i+1}\) for \(i=0,1,\dots,d\). Therefore, \([1,n]=[c_0,c_{d+1}]\) is covered by the (not necessarily maximal) cliques \([c_0,c_1],[c_1,c_2],\dots,[c_d,c_{d+1}]\). Since every such clique intersects \(\{c_1,\dots,c_d\}\), the latter subset is a reduced connected dominating set of \(G\). In particular, \(d\ge \gamma_c^{*}(G)\).

    Suppose for contradiction that \(d>\gamma_c^{*}(G)\). Then there exists a reduced connected dominating set \(D\) of \(G\) with \(\abs{D}<d\). In particular, there is a path \(c_0'=1,c_1',\dots,c_{d'}',c_{d'+1}'=n\) in \(G\) with \(c_1',\dots,c_{d'}'\in D\). It is clear that \(d'\le \abs{D}<d\). Since \(c_0'=1\) is a simplicial vertex, we must have \(c_1'\in F_1\). In particular, \(c_1'\le c_1\). We claim that \(c_2'\le c_2\).  Suppose for contradiction that \(c_2'>c_2\). Since \(c_1'\) and \(c_2'\) are connected, they are in a common maximal clique \(F\). Since \(c_1'\le c_1<c_2<c_2'\), \(c_1\) and \(c_2\) also belong to \(F\). But \(c_2\ne \max(F)\) due to the existence of \(c_2'\), which contradicts the choice of \(c_2\). Therefore, \(c_2'\le c_2\), as claimed. If we continue this reasoning, we will see that \(c'_{d'}\le c_{d'}\). Notice that \(c_{d'}'\) is connected to \(c_{d'+1}'=n\). Therefore, \(c'_{d'}\) and \(n\) are in a common maximal clique \(F'\). Since \(d>d'\), we see that \(c_{d'},c_{d'+1}\) and \(c_{d+1}=n\) are three distinct vertices in \(F'\). This contradicts the choice of \(c_{d'+1}\). Therefore, we must have \(d=\gamma_c^{*}(G)\).

    As a result, the ``in particular'' part of \Cref{lem:minimal_CDS} follows directly from \Cref{cor:v_empty}.
\end{proof}

Next, we consider the local \(\v\)-number of \(\calJ_{K_m,G}\) with respect to non-empty cut sets. The structure of those sets is explained below.

\begin{Lemma}
    [{\cite[Proposition 1.4]{MR4447411}}]
    \label{prop:cutsetsclosed}
    Let $G$ be a connected closed graph.
    Then, the following are equivalent for a non-empty set $T\subset [n]$:
    \begin{enumerate}[a]
        \item $T$ is a non-empty cut set of $G$;
        \item $T=W_{j_1}\sqcup \cdots \sqcup W_{j_s}$ for some $s\geq 1$ and $1\leq j_1<\cdots <j_s\leq t-1$, where $W_{j_i}$ is a connected cut set of $G$ for every $i$, and $\max(W_{j_i})+1<\min(W_{j_{i+1}})$ for $1\leq i\leq s-1$.
    \end{enumerate}
\end{Lemma}

It follows readily from this characterization that a non-empty set \(T\) is a connected cut set
if and only if it is a cut set and connected.

To obtain the local \(\v\)-number, we need to understand the colon ideals with respect to binomials associated with non-edges. We have discussed those ideals in \Cref{thm:colon_edge}. The result will be simpler, when \(G\) is a connected closed graph.

\begin{Lemma}
    \label{lem:colon_edge_short}
    Suppose that \(W_i = F_i \cap F_{i+1}\) is a connected cut set of a connected closed graph \(G\). Take arbitrary \(\varepsilon_1 \in E(K_m)\) and choose \(\varepsilon_2 = \{a_{i+1}-1, b_{i}+1\}\).  Then we have
    \[
        (\calJ_{K_m,G} : p_{(\varepsilon_1,\varepsilon_2)})
        =
        \calJ_{K_m,G\setminus W_i} + (x_{k,u} \mid k \in [m], u \in W_i).
    \]
\end{Lemma}
\begin{proof}
    Notice that \(\varepsilon_2\) is a non-edge of \(G\), and every vertex of \(W_i\) connects \(a_{i+1}-1\) and \(b_{i}+1\). On the other hand, if \(z_0=a_{i+1}-1,z_1,\dots,z_k,z_{k+1}=b_{i}+1\) is a path in \(G\) connecting \(a_{i+1}-1\) and \(b_{i}+1\), one of \(z_1,\dots,z_{k}\) must lie in \(W_i\). Therefore, it follows from \Cref{thm:colon_edge} that 
    \begin{align*}
        (\calJ_{K_m,G}:p_{(\varepsilon_1,\varepsilon_2)}) & =\calJ_{K_m,G_{\varepsilon_2}^\dag}+(x_{k,u}\mid k\in [m],u\in W_i)\\
        &=\calJ_{K_m,G_{\varepsilon_2}^\dag\setminus W_i}+(x_{k,u}\mid k\in [m],u\in W_i).
    \end{align*}

    Let \(v_1,v_2\in N_G(a_{i+1}-1)\) be two distinct vertices such that \(v_1,v_2\notin W_i\). It is easy to check that \(v_1,v_2\le b_{i}\). Since \(G\) is a closed graph, this implies that \(\{v_1,v_2\}\in E(G)\).

    Symmetrically, if \(v_1,v_2\in N_G(b_i+1)\) are two distinct vertices such that \(v_1,v_2\notin W_i\), then \(\{v_1,v_2\}\in E(G)\).

    Therefore, 
    \[
        \calJ_{K_m,G_{\varepsilon_2}^\dag\setminus W_i}+(x_{k,u}\mid k\in [m],u\in W_i)=\calJ_{K_m,G\setminus W_i} + (x_{k,u}\mid k\in [m],u\in W_i).
        \qedhere
    \]
\end{proof}

We will derive the local \(\v\)-number of \(\calJ_{K_m,G}\) through some combinatorially constructed homogeneous polynomials.

\begin{Construction}
    \label{construct:L_T_non_CM}
    Let \(T\) be a non-empty cut set of the connected closed graph \(G\) which has a decomposition as in \Cref{prop:cutsetsclosed}.

    For each adjacent connected cut set \(W_{j_i}=[a_{j_i+1},b_{j_i}]\) and \(W_{j_{i+1}}=[a_{j_{i+1}+1},b_{j_{i+1}}]\) with \(1\le i\le s-1\), if \(b_{j_i+1} \le a_{j_{i+1}}\), we set \(\beta_i=b_{j_i+1}\) and \(\alpha_{i+1}=a_{j_{i+1}}\). Otherwise, \(b_{j_i+1}> a_{j_{i+1}}\). In this case, since $\max(W_{j_i})+1<\min(W_{j_{i+1}})$, we set 
    \(\beta_i=\alpha_{i+1}=\min([a_{j_{i+1}},b_{j_{i}+1}]\setminus T)\).
    Furthermore, we set \(\alpha_1=a_{j_1}\) and \(\beta_s=b_{j_s+1}\).

    Moreover, we have \([n]\setminus \bigsqcup_{i=1}^s [\alpha_i+1,\beta_i-1]= \bigsqcup_{i=0}^s [\beta_i,\alpha_{i+1}]\), where we further set \(\beta_0=1\) and \(\alpha_{s+1}=n\). Each \([\beta_i,\alpha_{i+1}]\) is contained in precisely one connected component of $G\setminus T$, and the induced graph \(\widehat{G}_i\) of \(G\) on it is again a closed graph. Let \(C_i\) be a reduced connected dominating set of \(\widehat{G}_i\) with minimal cardinality, i.e.,  \(\abs{C_i}=\gamma_c^{*}(\widehat{G}_i)\). Since \(\beta_i\) and \(\alpha_{i+1}\) are simplicial vertices of \(\widehat{G}_i\), i.e., they are in unique maximal cliques respectively, we must have \(\{\beta_i,\alpha_{i+1}\}\cap C_i=\emptyset\).

    Now, we are ready to introduce a simple graph \(L(T)\), which is a subgraph of a path graph. The vertex set of \(L(T)\) is given by \(\{\alpha_i,\beta_i\mid i\in [s]\}\sqcup (\bigsqcup_{i=0}^s C_i)\). At the same time, the edge set of \(L(T)\) is \(\{\{\alpha_i,\beta_i\}\mid i\in [s]\}\).
    Since those sets \(C_i\) only contribute to the isolated vertices, this graph is completely determined, up to isomorphism, by the cut set $T$. 
\end{Construction}

The graph \(L(T)\) in \Cref{construct:L_T_non_CM} is a disjoint union of several (non-degenerate) path graphs, with some isolated vertices. As usual, every isolated vertex can be considered as a degenerate path graph \(P_1\) of length \(0\).

\begin{Example}
    \label{exam:closed_graph_non_CM}
    Consider the closed graph \(G\) on the vertex set \([42]\), with maximal cliques
    \begin{align*}
        [1,4],[3,9],[6,10],[9,13],[12,16],[15,19],[18,21],[20,23],\\
        [21,24], [22,25],[23,28],[27,30],[29,34],[33,37],[36,40],[39,42].
    \end{align*}
    According to \Cref{prop:cutsetsclosed},  
    \[
        T=[3,4]\cup[9,10]\cup[12,13]\cup[15,16]\cup [29,30]\cup [33,34] 
    \]
    is a cut set of $G$ with \(s=6\). Furthermore, from \Cref{construct:L_T_non_CM}, we know that $\beta_0=1=\alpha_1,\beta_1=6=\alpha_2,
    \beta_2=11=\alpha_3,
    \beta_3=14=\alpha_4,
    \beta_4=19,
    \alpha_5=27,
    \beta_5=31=\alpha_6,
    \beta_6=37$, and $\alpha_7=42$.
    Therefore,
    \([42]\setminus \bigsqcup_{i=1}^6 [\alpha_i+1,\beta_i-1]=\{1\}\cup\{6\}
    \cup\{11\}
    \cup\{14\}
    \cup[19,27]
    \cup\{31\}
    \cup[37,42]\).
    At the same time, \(C_0=C_1=C_2=C_3=C_5=\emptyset\),\(C_4=\{21,24\}\), and \(C_6=\{40\}\).
    Hence, the graph \(L(T)\) consists of two maximal paths on \(\{1,6,11,14,19\}\) and \(\{27,31,37\}\), respectively, and three isolated vertices \(\{21,24,40\}\); see also \Cref{fig:L(T)_non_CM}.
    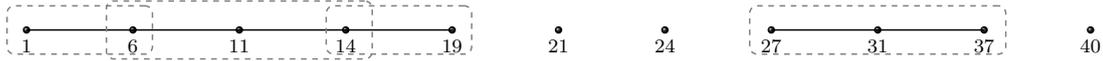
\begin{figure}[tbhp]
        \centering
        \scalebox{0.70}{
            \begin{tikzpicture}[thick, scale=1, every node/.style={scale=1.3}]
                \shade [shading=ball, ball color=black] (0,0) circle (.07) node [below] {\scriptsize$1$};
                \shade [shading=ball, ball color=black] (2,0) circle (.07) node [below] {\scriptsize$6$};
                \shade [shading=ball, ball color=black] (4,0) circle (.07) node [below] {\scriptsize$11$};
                \shade [shading=ball, ball color=black] (6,0) circle (.07) node [below] {\scriptsize$14$};
                \shade [shading=ball, ball color=black] (8,0) circle (.07) node [below] {\scriptsize$19$};
                \shade [shading=ball, ball color=black] (10,0) circle (.07) node [below] {\scriptsize$21$};
                \shade [shading=ball, ball color=black] (12,0) circle (.07) node [below] {\scriptsize$24$};
                \shade [shading=ball, ball color=black] (14,0) circle (.07) node [below] {\scriptsize$27$};
                \shade [shading=ball, ball color=black] (16,0) circle (.07) node [below] {\scriptsize$31$};
                \shade [shading=ball, ball color=black] (18,0) circle (.07) node [below] {\scriptsize$37$};
                \shade [shading=ball, ball color=black] (20,0) circle (.07) node [below] {\scriptsize$40$};

                \draw (0,0)--(8,0);
                \draw (14,0)--(18,0);

                \node [draw, rounded corners, fit={(2.5,0) (5.5,0)}, inner sep=12pt, dashed, gray] {};
                \node [draw, rounded corners, fit={(0.3,0) (1.7,0)}, inner sep=10pt, dashed, gray] {};
                \node [draw, rounded corners, fit={(6.3,0) (7.7,0)}, inner sep=10pt, dashed, gray] {};
                \node [draw, rounded corners, fit={(14.5,0) (17.5,0)}, inner sep=10pt, dashed, gray] {};
            \end{tikzpicture}
        }
        \caption{The associated graph $L(T)$}
        \label{fig:L(T)_non_CM}
    \end{figure}
\end{Example}

Next, we construct the homogeneous polynomial \(f_{\calP}\) from the graph \(L(T)\), which plays the role of \(f\) in the definition of the local \(\v\)-numbers.

\begin{Construction}
    \label{construct:f_P_non_CM}
    Consider a partition \(\calP\) of the edge set \(E(L(T))\), such that \(L(T)\) can be considered as a non-disjoint union of several shorter paths and isolated vertices. In this description, the union is only edge-disjoint, not vertex-disjoint. Furthermore, we require that each such path  has  length at most $m-1$. We shall call such a partition \emph{\(m\)-compatible}. Suppose that the set of vertices of the non-degenerate paths (which we shall call a \emph{slice}) after this partition is
    \[
        \{c_{i_p,1}<c_{i_p,2}<\cdots<c_{i_p,\ell_p}\}
    \]
    with \(2\le \ell_p\le m\) for each \(p\in [q]\), where $q$ is the number of slices.
    Furthermore, we introduce the polynomial
    \[
        f_{p}\coloneqq  \det \begin{pmatrix}
            x_{1,c_{i_p,1}} & x_{1,c_{i_{p},2}} & \cdots & x_{1,c_{i_{p},\ell_p}} \\
            x_{2,c_{i_p,1}} & x_{2,c_{i_{p},2}} & \cdots & x_{2,c_{i_{p},\ell_p}} \\
            \vdots & \vdots & & \vdots \\
            x_{\ell_p,c_{i_p,1}} & x_{\ell_p,c_{i_{p},2}} & \cdots & x_{\ell_p,c_{i_{p},\ell_p}} 
        \end{pmatrix}
    \]
    for each \(p\in [q]\), and set 
    \[
        {\bfm_i}\coloneqq \prod_{v\in C_i}x_{1,v}
    \]
    for each \(i=0,1,\dots,s\). 
    Note that if \(\widehat{G}_i\) is a complete graph, then \(C_i=\emptyset\) and \(\bfm_i=1\).

    Then, the expected polynomial associated to this partition is
    \[
        f_{\calP}\coloneqq 
        \left(\prod_{p=1}^q f_p\right)\cdot \left(\prod_{i=0}^s\bfm_i\right)=
        \left(\prod_{p=1}^q f_p\right)\cdot \left(\prod_{\substack{\text{\(v\) is an isolated}\\ \text{vertex of \(L(T)\)}}}x_{1,v}\right).
    \]
\end{Construction}

\begin{Example}
    Consider the graph \(L(T)\) in \Cref{exam:closed_graph_non_CM}. The following partition \(\calP\) of its edge set
    \[
        E(L(T))=\{\{1,6\}\}\sqcup \{\{6,11\},\{11,14\}\}\sqcup\{\{14,19\}\}\sqcup\{\{27,31\},\{31,37\}\}
    \]
    is \(3\)-compatible; see also \Cref{fig:L(T)_non_CM}. Thus, there are \(4\) slices with respect to this partition.  Since the isolated vertices of \(L(T)\) are \(21, 24, 40\), the initial monomial of the polynomial \(f_{\calP}\) is given by
    \[
        (x_{1,1}x_{2,6})(x_{1,6}x_{2,11}x_{3,14})(x_{1,14}x_{2,19})x_{1,21}x_{1,24}(x_{1,27}x_{2,31}x_{3,37})x_{1,40}.
    \]
\end{Example}

In the following, we show that the polynomial just defined fulfills the  required role with respect to the colon ideals.

\begin{Lemma}
    \label{lem:closed_f_works}
    The  polynomial \(f_{\calP}\) defined in \Cref{construct:f_P_non_CM} satisfies 
    \begin{equation}
        (\calJ_{K_m,G}:f_{\calP})=\frakP_T(K_m,G).
        \label{eqn:closed_f_works}
    \end{equation}
\end{Lemma}
\begin{proof}
    Since \(\frakP_T(K_m,G)\) is a prime ideal containing \(\calJ_{K_m,G}\), it suffices to show that \(f_{\calP}\notin \frakP_T(K_m,G)\) and \(f_{\calP}\frakP_T(K_m,G)\subseteq \calJ_{K_m,G}\). Once we have shown these, we will have
    \[
        \frakP_T(K_m,G)\subseteq (\calJ_{K_m,G}:f_{\calP})\subseteq (\frakP_{T}(K_m,G):f_{\calP})=\frakP_T(K_m,G),
    \]
    which implies that \((\calJ_{K_m,G}:f_{\calP})=\frakP_T(K_m,G)\).

    First,
    we show that \(f_{\calP}\notin \frakP_T(K_m,G)\). Suppose for
    contradiction that \(f_{\calP}\in \frakP_T(K_m,G)\). Since the linear
    generators of \(\frakP_T(K_m,G)\) take the form \(x_{i,v}\) with \(i\in[m]\)
    and \(v\in T\), while the ``support'' \(C_i\) of each \(\bfm_i\) does not 
    intersect
    \(T\), this implies that \(\prod_{0\le i\le s}\bfm_i\notin
    \frakP_T(K_m,G)\). Therefore, \(f_p\in \frakP_T(K_m,G)\) for some \(p\in
    [q]\). In particular, the initial monomial
    \(\ini(f_p)=x_{1,c_{i_p,1}}x_{2,c_{i_{p},2}}\cdots
    x_{\ell_p,c_{i_{p},\ell_p}}\) belongs to \(\ini(\frakP_T(K_m,G))\). This
    implies that either some \(x_{i,u}\) divides \(\ini(f_p)\) with \(i\in
    [m]\) and \(u\in T\), or some \(x_{i,u}x_{j,v}\) divides
    \(\ini(f_{p})\) with \(u,v\) in the same connected component of
    \(G\setminus T\) and \(i\ne j\) in \([m]\). However, since the vertices in
    \(\{c_{i_p,1},c_{i_p,2},\dots,c_{i_p,\ell_p}\}\) belong to pairwise
    different connected components of \(G\setminus T\), this is impossible.
    Therefore, \(f_{\calP}\notin \frakP_T(K_m,G)\).

    Second, we show that \(f_{\calP}\frakP_T(K_m,G)\subseteq \calJ_{K_m,G}\).
    We will take an arbitrary canonical generator \(g\) of
    \(\frakP_{T}(K_m,G)\), and show that \(g f_{\calP} \in \calJ_{K_m,G}\).
    There are two cases.
    \begin{enumerate}[i]
        \item Suppose that \(g=x_{k_0,v}\) with \(k_0\in [m]\) and \(v\in W_{j_i}\subseteq T\). For this index $i$, since \(\{\alpha_i,\beta_i\}\in E(L(T))\),
            \(f_{\calP}\) is an \(S\)-linear combination of
            \([k_1,k_2|\alpha_i,\beta_i]\) with \(1\le k_1<k_2 \le m\), due to
            the Laplace Expansion Theorem for determinants. Notice that \(\{\alpha_i,\beta_i\}\in E(G_{v})\).
            It follows from \Cref{thm:colon_x} that
            \(x_{k_0,v}[k_1,k_2|\alpha_i,\beta_i]\in \calJ_{K_m,G}\).
            Therefore, \(g f_{\calP}\in \calJ_{K_m,G}\).
        \item Suppose that \(g=[k_1,k_2|u,v]\) where \(1\le k_1<k_2\le m\), and
            \(u,v\) are two distinct vertices of the same connected component,
            say, \(G'\), of \(G\setminus T\). Notice that \(G'\) is still a
            connected closed graph. We shall show that
            \(f_{\calP}[k_1,k_2|u,v]\in \calJ_{K_m,G}\). Without loss of
            generality, we may assume that \(u<v\) and \(\{u,v\}\notin E(G)\),
            i.e., \(\{u,v\}\notin E(G')\). Note that the isolated vertices of
            \(L(T)\) in \(G'\) form a reduced connected dominating set of \(G'\). Thus,
            the vertices \(u\) and \(v\) are connected via the path
            \(u=\tau_0,\tau_1,\dots,\tau_s,\tau_{s+1}=v\) in \(G'\), where
            \(\tau_1,\dots,\tau_{s}\) are all isolated vertices in \(L(T)\).
            Notice that \(f_{\calP}\) is an \(S\)-linear combination of the
            monomials \(x_{h_1,\tau_1}\cdots x_{h_s,\tau_s}\) with
            $h_1,\dots,h_s\in [m]$. Furthermore, \(x_{h_1,\tau_1}\cdots
            x_{h_s,\tau_s} [k_1,k_2|u,v]\in \calJ_{K_m,G'}\subseteq
            \calJ_{K_m,G}\), by \Cref{cor:colon}. Therefore,
            \(f_{\calP}[k_1,k_2|u,v]\in \calJ_{K_m,G}\). 
    \end{enumerate} 
    Since we have verified the two expected conditions, we conclude that equality holds in \eqref{eqn:closed_f_works}.
\end{proof}

The following result, which is the main result of this section, computes the local \(\v\)-number of \(\calJ_{K_m,G}\) with \(m=2\).
Notice that every edge of \(L(T)\) has cardinality \(2\).
Therefore, every \(2\)-compatible partition of \(E(L(T))\) is trivial: in this partition, every equivalence class contains precisely one edge.
Let \(\calP_0\) be this unique \(2\)-compatible partition.
By abuse of notation, we write \(f_T\) for the polynomial \(f_{\calP_0}\).
We will show that the local \(\v\)-number of \(\calJ_{K_2,G}\) is determined by the degree of the polynomial \(f_{T}\).

Recall that with respect to the lexicographic order on the polynomial ring \(S\), the initial ideal of \(\calJ_{K_2,G}\) is given by
\[
(x_{1,k}x_{2,l} \mid 1\le k<l\le n \text{ such that } \{k,l\}\in E(G)),
\]
as shown in \cite[Theorem 1.3]{MR3290687}.

\begin{Theorem}
    \label{thm:local_v_number_non_CM_closed}
    The local \(\v\)-number with respect to the cut set \(T\) of \(G\) is given by \(\deg(f_T)\).
\end{Theorem}
\begin{proof}
    Let \(f\) be a homogeneous polynomial in \(S\) such that
    \((\calJ_{K_2,G}:f)=\frakP_T(K_2,G)\). 
    By \Cref{lem:closed_f_works}, it
    suffices to show that 
    \begin{equation}
        \deg(f_{T})\le \deg(f).
        \label{eqn:f_ft}
    \end{equation}
    Furthermore, by \Cref{lem:change_initial}, we may assume without loss of generality that \(\ini(f)\notin \ini(\calJ_{K_2,G})\).  To establish the inequality in \eqref{eqn:f_ft}, it is critical to observe that the factors of \(\ini(f)\) stem from two distinct cases, which form the core of our analysis.

    \begin{enumerate}[a]
        \item Consider the induced subgraph \(\widehat{G}_i\) of $G$ defined on the vertex set \([\beta_i,\alpha_{i+1}]\), where \(0\le i\le s\). As noted earlier, this subgraph is also a connected closed graph. Let \(\widehat{F}_1,\dots,\widehat{F}_p\) denote all the maximal cliques of \(\widehat{G}_i\). Note that the restriction of any maximal clique of \(G\) to \(\widehat{G}_i\) forms a clique of \(\widehat{G}_i\), though not necessarily a maximal one.  Without loss of generality, we may assume that \(\min(\widehat{F}_1) <\min(\widehat{F}_2) <\cdots<\min(\widehat{F}_p)\).

            The vertex \(\widehat{c}_0\coloneqq \beta_i\) is a simplicial vertex and only belongs to the maximal clique \(\widehat{F}_1\). We will set \(\widehat{c}_1\coloneqq \max(\widehat{F}_1)\). In general, if \(\widehat{c}_k\) has been defined and \(\widehat{c}_k\ne \alpha_{i+1}\), set \(\widehat{c}_{k+1}\coloneqq \max\{\max(\widehat{F}_j)\mid \widehat{c}_k\in \widehat{F}_j\}\).  It follows from \Cref{lem:minimal_CDS} that the path \(\widehat{c}_0= \beta_i,\widehat{c}_1,\dots, \widehat{c}_{\gamma_c^{*}(\widehat{G}_i)+1}=\alpha_{i+1}\)  is the shortest path connecting   
            \(\beta_i\) and \(\alpha_{i+1}\).  Furthermore, \(\{\widehat{c}_1, \dots, \widehat{c}_{\gamma_c^{*}(\widehat{G}_i)}\}\) is a reduced connected dominating set of minimal cardinality of \(\widehat{G}_i\). 

            For each \(\widehat{c_k}\) with \(1\le k\le \gamma_c^{*}(\widehat{G}_i)\), suppose that \(\widehat{c}_k=\max(\widehat{F}_{j_k})\). Let \(\widehat{f}_k=p_{(\{1,2\},\varepsilon_k)}\) be the binomial associated to the pair 
            \(\{1,2\}\in E(K_2)\) and \(\varepsilon_k=\{\min(\widehat{F}_{j_k+1})-1, \max(\widehat{F}_{j_k})+1\}\in E(H)\), where $H$   is the complete graph  on the union of the vertex sets of $\widehat{F}_{j_k}$ and  $\widehat{F}_{j_k+1}$. 
            Set \(\widehat{W}_k\coloneqq [\min(\widehat{F}_{j_k+1}), \max(\widehat{F}_{j_k})]\); this set is a connected cut set of \(\widehat{G}_i\). Clearly, \(\widehat{W}_k\) is also a cut set of \(G\), and hence a connected cut set of \(G\). It follows from \Cref{lem:colon_edge_short} that 
            \[
                (\calJ_{K_2,G}:\widehat{f}_k)=  \calJ_{K_2,G\setminus \widehat{W}_k}+(x_{j,u}\mid j\in[2],u\in \widehat{W}_k).
            \]
            Since the vertices of \(\varepsilon_k\) belong to \(\widehat{G}_i\), a connected component of \(G\setminus T\), \(\widehat{f}_k\in \frakP_T(K_2,G)\).  As a result, 
            \(f\in  (\calJ_{K_2,G}:\frakP_T(K_2,G)) \subseteq (\calJ_{K_2,G}:\widehat{f}_k)\), which in turn implies that 
            \[
                \ini(f)\in \ini(\calJ_{K_2,G\setminus \widehat{W}_k})+(x_{j,u}\mid j\in[2],u\in \widehat{W}_k).
            \]
            Given that \(\ini(\calJ_{K_2,G\setminus \widehat{W}_k})\subseteq
            \ini(\calJ_{K_2,G})\),  while \(\ini(f)\notin
            \ini(\calJ_{K_2,G})\), it follows that \(\ini(f)\) is divisible
            by some \(z_k\in \{x_{j,u}\mid j\in[2], u\in \widehat{W}_k\}\).

            Notice that by the maximality condition defining \(\widehat{c}_{k}\), we have
            \[
                \min(\widehat{F}_{j_k})\le
                \widehat{c}_{k-1}=\max(\widehat{F}_{j_{k-1}})<\min(\widehat{F}_{j_k+1})\le
                \min(\widehat{F}_{j_{k+1}}).
            \]
            Therefore,
            \(\widehat{W}_{1},\dots,\widehat{W}_{\gamma_c^{*}(\widehat{G}_i)}\) are
            mutually disjoint. This implies that \(\ini(f)\) is divisible by
            \(\widehat{\bfm}_i\coloneqq \lcm\{z_k\mid k\in
            [\gamma_c^{*}(\widehat{G}_i)]\}=\prod_{k\in
        \gamma_c^{*}(\widehat{G}_i)}z_k\). Notice that
        \(\deg(\widehat{\bfm}_i)=\gamma_c^{*}(\widehat{G}_i)=\deg(\bfm_i)\), and
        \(\bigcup_{k\in [\gamma_c^{*}(\widehat{G}_i)]} \widehat{W}_{k}\subseteq
        [\beta_i+1,\alpha_{i+1}-1]\).

    \item 
        Let \(W_{j_i}\subseteq T\) be a connected cut set. For every vertex \(v\in W_{j_i}\), since \(x_{1,v}\in \frakP_T(K_2,G)\), we have \(f\in (\calJ_{K_2,G}:x_{1,v})=\calJ_{K_2,G_{v}}\) by \Cref{thm:colon_x}.  Notice that \(G_{v}\) is a connected closed graph with \(E(G_{v})=E(G)\cup\{\{w',w''\}\mid w',w''\in N_G(v)\}\). Given that \(\ini(f)\notin \ini(\calJ_{K_2,G})\), we infer that there exist suitable \(w_{i,v}',w_{i,v}''\in N_G(v)\) such that \(\{w_{i,v}',w_{i,v}''\}\notin E(G)\) and \(\ini(f)\) is divisible by \(x_{1,w_{i,v}'}x_{2,w_{i,v}''}\).  We claim that neither \(w_{i,v}'\) nor \(w_{i,v}''\) belongs to \(W_{j_i}\) for all \(v\in W_{j_i}\). 

        Suppose for contradiction that \(w_{i,v}' \in W_{j_i}\) for some \(v \in W_{j_i}\).  Let \(\widetilde{v} = w_{i,v}'\). Then, \(\ini(f)\) is divisible by \(x_{1,\widetilde{v}}\).  Since \(\widetilde{v} \in W_{j_i}\), the elements \(w_{i,\widetilde{v}}'\) and \(w_{i,\widetilde{v}}''\) are also defined, with \(\widetilde{v} < w_{i,\widetilde{v}}''\) and \(\ini(f)\) divisible by \(x_{2,w_{i,\widetilde{v}}''}\).  In particular, \(\ini(f)\) is divisible by \(x_{1,\widetilde{v}}x_{2,w_{i,\widetilde{v}}''}\).  As \(\widetilde{v}\) is adjacent to \(w_{i,\widetilde{v}}''\) in \(G\) and \(\ini(f)\) is divisible by \(x_{1,\widetilde{v}}x_{2,w_{i,\widetilde{v}}''}\), we have \(\ini(f) \in \ini(\calJ_{K_2,G})\), a contradiction. Thus the claim holds, and neither \(w_{i,v}'\) nor \(w_{i,v}''\) belongs to \(W_{j_i}\).

        Set \(w_i'\coloneqq \max\{w_{i,v}'\mid v\in W_{j_i}\}=w_{i,v_1}'\) and \(w_i''\coloneqq\min\{w_{i,v}''\mid v\in W_{j_i}\}=w_{i,v_2}''\) with $v_1,v_2\in W_{j_i}$.
        It follows that
        \(\ini(f)\) is divisible by 
        both $x_{1,w_{i,v_1}'}x_{2,w_{i,v_1}''}$ and $x_{1,w_{i,v_2}'}x_{2,w_{i,v_2}''}$. Therefore, $\ini(f)$ is divisible by
        \(x_{1,w_{i,v_1}'}x_{2,w_{i,v_2}''}=x_{1,w_i'}x_{2,w_i''}\).  Furthermore, it
        follows from the discussion in the last paragraph that
        \(w_i'<\min(W_{j_i})\le\max(W_{j_i})<w_i''\).  Notice that when
        \(v=\max(W_{j_i})\), it is easy to see that \(w_{i,v}'\ge
        \min(F_{j_i})\). Hence, \(w_i'\ge \min(F_{j_i})\).
        Similarly, \(w_i''\le \max(F_{j_i+1})\). 

        We claim in addition that \(w_i'>\max(W_{j_{i-1}})\). Suppose for
        contradiction that \(w_i'\le \max(W_{j_{i-1}})\). Notice that
        \(\ini(f)\) is divisible by \(x_{1,w_i'}x_{2,w_{i-1}''}\), and
        \(\min(F_{j_i})\le w_i'\le \max(W_{j_{i-1}})<w_{i-1}''\le
        \max(F_{j_{i-1}+1})\le \max(F_{j_i})\). This implies that \(w_i'\) and
        \(w_{i-1}''\) are two distinct vertices in the common clique
        \(F_{j_i}\). In particular, \(\ini(f)\in \ini(\calJ_{K_2,G})\), a
        contradiction. Similarly, we will have \(w_i''<\min(W_{j_{i+1}})\).

        In short, we have 
        \begin{equation}
            \max(W_{j_{i-1}})<w_i'<\min(W_{j_i})\le\max(W_{j_i})<w_i''<\min(W_{j_{i+1}}).
            \label{eqn:ws_positions}
        \end{equation}
    \end{enumerate}

    Let us turn to the graph \(L(T)\) in \Cref{construct:L_T_non_CM}.  Recall that \(L(T)\) is a disjoint union of several non-degenerate path graphs and some isolated vertices. In the following arguments, any such non-degenerate path in this description will be {called} \emph{maximal}. 
    Take a maximal path \(\bfP\) in \(L(T)\) and suppose that the vertices of \(\bfP\) are
    \begin{equation*}
        \alpha_{\xi_1}<\beta_{\xi_1}=\alpha_{\xi_2}<\cdots<\beta_{\xi_{\ell-1}}=\alpha_{\xi_{\ell}}<\beta_{\xi_{\ell}}.
    \end{equation*}
    It follows from \eqref{eqn:ws_positions} that
    \(\ini(f)\) is also divisible by 
    \[
        f_{\bfP}\coloneqq \lcm \Set{x_{1,w_{\xi_i}'}x_{2,w_{\xi_i}''} | 1\le i\le \ell}=\prod_{i=1}^{\ell}x_{1,w_{\xi_i}'}x_{2,w_{\xi_i}''}.
    \]
    Furthermore, the support of \(f_{\bfP}\) is a subset of \([\min(W_{j_{\xi_1}}),\max(W_{j_{\xi_{\ell}}})]=[\alpha_{\xi_1},\beta_{\xi_{\ell}}]\).

    To summarize, \(\ini(f)\) is divisible by
    \begin{align}
        f^*&\coloneqq \lcm\left(\left\{f_{\bfP}\mid \text{\(\bfP\) is a {maximal} path in \(L(T)\)}\right\}\cup\left\{\widehat{\bfm}_i\mid 0\le i\le s\right\}\right) \label{eqn:f_star_disojoint}\\
        &=\left( \prod_{\bfP\text{ is a {maximal} path}} f_{\bfP}\right) \cdot \left(\prod_{0\le i\le s}\widehat{\bfm}_i \right), \notag
    \end{align}
    where the equality holds since the supports of the squarefree monomials in
    \eqref{eqn:f_star_disojoint} are pairwise disjoint. 

    Notice that 
    \[
        \deg\left(\prod_{0\le i\le s}\widehat{\bfm}_i \right) =\sum_{i=0}^s\deg (\widehat{\bfm}_i) = \sum_{i=0}^s\deg ({\bfm}_i),
    \]
    Furthermore, the quadratic monomial \(x_{1,w_{\xi_i}'}x_{2,w_{\xi_i}''}\) in the definition of \(f_{\bfP}\) corresponds to the edge \(\{\alpha_{\xi_i},\beta_{\xi_i}\}\). Thus, \(\deg(f^{*})=\deg(f_T)\).
    Since \(f^{*}\) divides \(\ini(f)\), this implies that \(\deg(f)\ge \deg(f_{T})\), as expected.
\end{proof}

It is natural to ask if \Cref{thm:local_v_number_non_CM_closed} holds for generalized binomial edge ideals, i.e., do we have
\[
    \v_{\mathfrak{P}_T}(\calJ_{K_m,G}) = \min\{\deg(f_{\calP}) \mid \calP \text{ is an } m\text{-compatible partition of } E(L(T))\}
\]
when \(G\) is a connected closed graph?  Furthermore, we may seek a closed-form formula for the \(\v\)-number of \(\mathcal{J}_{K_m,G}\).  However, this task is non-trivial even for \(m=2\), owing to the highly interlaced structure of the maximal cliques of \(G\).  The main obstruction stems from the difficulty of identifying the cut set \(T\) that attains the minimal local \(\v\)-number.  For this reason, in the next section we impose the additional condition that \(G\) is Cohen--Macaulay.

\section{$\v$-number of generalized binomial edge ideals of Cohen--Macaulay closed graphs}
\label{sec:closegraph}

In this section, let $G$ be a connected Cohen-Macaulay closed graph and let \(m\ge 2\) be an integer. We will study the \(\v\)-number of the generalized binomial edge ideal \(\calJ_{K_m,G}\).

It follows from \cite[Theorem 3.1]{MR2863365} that we may assume there exist integers \(1 = a_1 < a_2 < \cdots < a_t < a_{t+1} = n\) such that the maximal cliques of \(G\) are \(F_i = [a_i,a_{i+1}]\) for \(i = 1,\dots,t\).  In other words, using the notation in \Cref{set:closed_graph_non_CM}, we have \(b_i = a_{i+1}\) for all \(i \in [t]\). We will also denote \(a_1\) by \(b_0\).
Consequently, every connected cut set \(W_i=[a_i,b_i]\cap[a_{i+1},b_{i+1}]=\{b_i\}\) of \(G\) is a singleton.

We will refer to the path graph with vertex set \(A\coloneqq \{b_0,b_1,\dots,b_t\} \) as the \emph{spine} of \(G\).  Furthermore, let \(\widetilde{C}(G)\coloneqq \{b_1,\dots,b_{t-1}\} \) denote the set of cut vertices of \(G\).  By \cite[Proposition 3.5]{MR2863365} or \Cref{prop:cutsetsclosed}, every cut set of \(G\) is contained in \(\widetilde{C}(G)\).  Moreover, if \(T= \{b_{j_1}<b_{j_2}<\cdots<b_{j_s}\} \subseteq \widetilde{C}(G)\) is non-empty, then \(T\in \calC(G)\) if and only if \(b_{j_i}+2\le b_{j_{i+1}}\) for all \(i\in [s-1]\).

\begin{Example}
    \label{exam1}
    The graph depicted in \Cref{fig1} is a Cohen--Macaulay closed graph with $t=14$ maximal cliques.  The vertices on the spine are
    \[
        \{b_0,b_1,\dots,b_{14}\}
        =\{1,3,6,7,9,12,13,15,18,19,21,22,24,26,27\}.
    \]   
    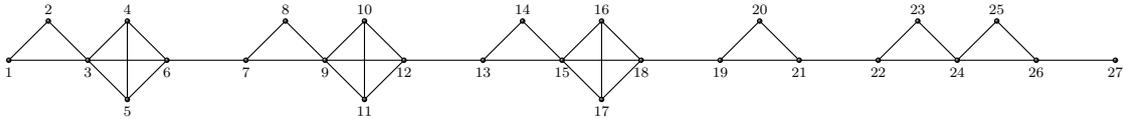
\begin{figure}[tbhp]
        \centering
        \scalebox{0.52}{
            \begin{tikzpicture}[thick, scale=1, every node/.style={scale=1.3}]
                \draw (0,0)--(28,0);
                \draw (0,0)--(1,1)--(2,0);
                \draw (2,0)--(3,1)--(4,0)--(3,-1)--(2,0);
                \draw (3,1)--(3,-1);
                \draw (6,0)--(7,1)--(8,0);
                \draw (9,1)--(10,0)--(9,-1)--(8,0)--(9,1)--(9,-1);
                \draw (12,0)--(13,1)--(14,0);
                \draw (15,1)--(16,0)--(15,-1)--(14,0)--(15,1)--(15,-1);
                \draw (18,0)--(19,1)--(20,0);
                \draw (22,0)--(23,1)--(24,0);
                \draw (24,0)--(25,1)--(26,0);

                \shade [shading=ball, ball color=black] (0,0) circle (.07) node [below] {\scriptsize$1$};
                \shade [shading=ball, ball color=black] (1,1) circle (.07) node [above] {\scriptsize$2$};
                \shade [shading=ball, ball color=black] (2,0) circle (.07) node [below] {\scriptsize$3$};
                \shade [shading=ball, ball color=black] (3,1) circle (.07) node [above] {\scriptsize$4$};
                \shade [shading=ball, ball color=black] (3,-1) circle (.07) node [below] {\scriptsize$5$};
                \shade [shading=ball, ball color=black] (4,0) circle (.07) node [below] {\scriptsize$6$};
                \shade [shading=ball, ball color=black] (6,0) circle (.07) node [below] {\scriptsize$7$};
                \shade [shading=ball, ball color=black] (7,1) circle (.07) node [above] {\scriptsize$8$};
                \shade [shading=ball, ball color=black] (8,0) circle (.07) node [below] {\scriptsize$9$};
                \shade [shading=ball, ball color=black] (9,1) circle (.07) node [above] {\scriptsize$10$};
                \shade [shading=ball, ball color=black] (9,-1) circle (.07) node [below] {\scriptsize$11$};
                \shade [shading=ball, ball color=black] (10,0) circle (.07) node [below] {\scriptsize$12$};
                \shade [shading=ball, ball color=black] (12,0) circle (.07) node [below] {\scriptsize$13$};
                \shade [shading=ball, ball color=black] (13,1) circle (.07) node [above] {\scriptsize$14$};
                \shade [shading=ball, ball color=black] (14,0) circle (.07) node [below] {\scriptsize$15$};
                \shade [shading=ball, ball color=black] (15,1) circle (.07) node [above] {\scriptsize$16$};
                \shade [shading=ball, ball color=black] (15,-1) circle (.07) node [below] {\scriptsize$17$};
                \shade [shading=ball, ball color=black] (16,0) circle (.07) node [below] {\scriptsize$18$};
                \shade [shading=ball, ball color=black] (18,0) circle (.07) node [below] {\scriptsize$19$};
                \shade [shading=ball, ball color=black] (19,1) circle (.07) node [above] {\scriptsize$20$};
                \shade [shading=ball, ball color=black] (20,0) circle (.07) node [below] {\scriptsize$21$};
                \shade [shading=ball, ball color=black] (22,0) circle (.07) node [below] {\scriptsize$22$};
                \shade [shading=ball, ball color=black] (23,1) circle (.07) node [above] {\scriptsize$23$};
                \shade [shading=ball, ball color=black] (24,0) circle (.07) node [below] {\scriptsize$24$};
                \shade [shading=ball, ball color=black] (25,1) circle (.07) node [above] {\scriptsize$25$};
                \shade [shading=ball, ball color=black] (26,0) circle (.07) node [below] {\scriptsize$26$};
                \shade [shading=ball, ball color=black] (28,0) circle (.07) node [below] {\scriptsize$27$};
            \end{tikzpicture}
        }
        \caption{A Cohen--Macaulay closed graph}
        \label{fig1}
    \end{figure}
\end{Example}

\subsection{Local $\v$-number}

Let \(T = \{b_{j_1} < b_{j_2} < \cdots < b_{j_s}\}\) be a cut set of \(G\).  To better characterize the local \(\v\)-number \(\v_T(\calJ_{K_m,G})\), we constructed a simple graph \(L(T)\) in \Cref{sec:nonCM_closed_graph}.  Since \(G\) is Cohen--Macaulay in this section, the description of \(L(T)\) can be simplified slightly.

\begin{Construction}
    \label{construct:L_T}
    Let \(A\coloneqq \{b_0,b_1,\dots,b_t\}\) be the set of vertices of the spine of $G$. For the given cut set $T$, let
    \[
        V_0\coloneqq \begin{cases}
            A\setminus T, & \text{if $\{b_1,b_{t-1}\}\cap T=\{b_1,b_{t-1}\}$},\\
            A\setminus (T\cup \{b_{t}\}), & \text{if $\{b_1,b_{t-1}\}\cap T=\{b_1\}$},\\
            A\setminus (T\cup \{b_{0}\}), & \text{if $\{b_1,b_{t-1}\}\cap T=\{b_{t-1}\}$},\\
            A\setminus (T\cup \{b_{0},b_{t}\}), & \text{if $\{b_1,b_{t-1}\}\cap T=\emptyset$}.
        \end{cases}
    \]

    For each adjacent vertices \(b_{j_i}\) and \(b_{j_{i+1}}\) with \(1\le i\le s-1\), if \({j_i+1} <{j_{i+1}}\), we set \(\beta_i=b_{j_i+1}\) and \(\alpha_{i+1}=b_{j_{i+1}-1}\). Otherwise, \({j_i+1} ={j_{i+1}}\). In this case, since $b_{j_i}+1<b_{j_{i+1}}$, we set  \(\beta_i=\alpha_{i+1}=b_{j_i}+1\). 
    Furthermore, we set \(\alpha_1=b_{j_1-1}\) and \(\beta_s=b_{j_s+1}\).

    The vertex set of \(L(T)\) is defined as \(V(L(T)) \coloneqq (V_0\setminus \bigcup_{i=1}^s [\alpha_i,\beta_i]) \sqcup \left(\bigcup_{i=1}^s \{\alpha_i,\beta_i\}\right)\).
The edge set of \(L(T)\) is defined as \(E(L(T)) \coloneqq \{\{\alpha_i,\beta_i\} \mid i \in [s]\}\).
This graph is completely determined by the cut set \(T\).
\end{Construction}

It is straightforward to verify that the constructions in \Cref{construct:L_T} and \Cref{construct:L_T_non_CM} are identical when \(G\) is a connected Cohen--Macaulay closed graph.

Recall that in \Cref{sec:nonCM_closed_graph}, we further considered an \(m\)-compatible partition \(\calP\) of \(E(L(T))\). From this partition, we introduced a homogeneous polynomial \(f_{\calP}\), which plays a role in computing the local \(\v\)-number \(\v_{T}(\calJ_{K_m,G})\).

\begin{Example}
    Let $G$ be the Cohen--Macaulay closed graph in \Cref{exam1}. For the cut set 
    \[
        T=\{b_1=3,b_2=6,b_4=9,b_8=18,b_{10}=21\},
    \]
    the associated graph $L(T)$ has vertex set
    \begin{align*}
        V(L(T))=\{1,4,7,12,13,15,19,22,24,26\}
    \end{align*}
    and edge set
    \begin{equation}
        E(L(T))=\{\uline{\{1,4\},\{4,7\}},\uline{\{7,12\}},\uline{\{15,19\},\{19,22\}}\}.
        \label{eqn:E(L(T))_underlines}
    \end{equation}
    As shown in \Cref{fig2}, $L(T)$ is the disjoint union of $2$ path graphs, with $3$ extra isolated vertices. The underlines in \eqref{eqn:E(L(T))_underlines} give a $3$-compatible partition $\calP$ on $E(L(T))$. There exists $3$ slices with respect to this partition. The initial monomial of the polynomial $f_{\calP}$ is
    \[
        (x_{1,1}x_{2,4}x_{3,7})(x_{1,7}x_{2,12})x_{1,13}(a_{1,15}x_{2,19}x_{3,22})x_{1,24}x_{1,26}.
    \]
    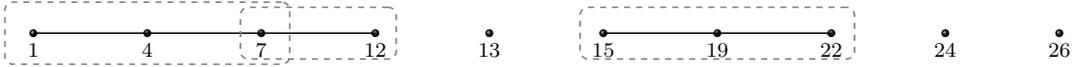
\begin{figure}[tbhp]
        \centering
        \scalebox{0.75}{
            \begin{tikzpicture}[thick, scale=1, every node/.style={scale=1.3}]
                \shade [shading=ball, ball color=black] (0,0) circle (.07) node [below] {\scriptsize$1$};
                \shade [shading=ball, ball color=black] (2,0) circle (.07) node [below] {\scriptsize$4$};
                \shade [shading=ball, ball color=black] (4,0) circle (.07) node [below] {\scriptsize$7$};
                \shade [shading=ball, ball color=black] (6,0) circle (.07) node [below] {\scriptsize$12$};
                \shade [shading=ball, ball color=black] (8,0) circle (.07) node [below] {\scriptsize$13$};
                \shade [shading=ball, ball color=black] (10,0) circle (.07) node [below] {\scriptsize$15$};
                \shade [shading=ball, ball color=black] (12,0) circle (.07) node [below] {\scriptsize$19$};
                \shade [shading=ball, ball color=black] (14,0) circle (.07) node [below] {\scriptsize$22$};
                \shade [shading=ball, ball color=black] (16,0) circle (.07) node [below] {\scriptsize$24$};
                \shade [shading=ball, ball color=black] (18,0) circle (.07) node [below] {\scriptsize$26$};

                \draw (0,0)--(6,0);
                \draw (10,0)--(14,0);

                \node [draw, rounded corners, fit={(0.5,0) (3.5,0)}, inner sep=12pt, dashed, gray] {};
                \node [draw, rounded corners, fit={(4.3,0) (5.7,0)}, inner sep=10pt, dashed, gray] {};
                \node [draw, rounded corners, fit={(10.5,0) (13.5,0)}, inner sep=10pt, dashed, gray] {};
            \end{tikzpicture}
        }
        \caption{The associated graph $L(T)$}
        \label{fig2}
    \end{figure}
\end{Example}

Next, we generalize \Cref{thm:local_v_number_non_CM_closed} to the case of generalized binomial edge ideals where \(G\) is a Cohen--Macaulay closed graph.
Specifically, we show that the minimal degree of the polynomials \(f_{\calP}\) gives the desired local \(\v\)-number with respect to \(T\).

Recall that with respect to the lexicographic order on the polynomial ring \(S\), the initial ideal of \(\calJ_{K_m,G}\) is given by
\[
(x_{k_1,u}x_{k_2,v} \mid 1 \le k_1 < k_2 \le m \text{ and } 1 \le u < v \le n \text{ with } \{u, v\} \in E(G)),
\]
as shown in \cite[Theorem 1.3]{MR3290687}.

\begin{Proposition}
    \label{prop:local_v_number_path}
    The local \(\v\)-number with respect to the cut set \(T\) of \(G\) is given by
    \[
        \v_{\frakP_T}(\calJ_{K_m,G})=\min\{\deg(f_{\calP})\mid \text{\(\calP\) is an \(m\)-compatible partition of \(E(L(T))\)}\}.
    \]
\end{Proposition}
\begin{proof}
    Let \(f\) be a homogeneous polynomial in \(S\) such that \((\calJ_{K_m,G} : f) = \mathfrak{P}_T(K_m,G)\).  It follows from \Cref{lem:closed_f_works} that it suffices to show the existence of an \(m\)-compatible partition \(\calP\) of \(E(L(T))\) satisfying \(\deg(f_{\calP}) \le \deg(f)\), i.e., \(\deg(f_{\calP}) \le \deg(\ini(f))\).  By \Cref{lem:change_initial}, we may assume without loss of generality that \(\ini(f) \notin \ini(\calJ_{K_m,G})\).
Two cases contribute to the factors of \(\ini(f)\).
    \begin{enumerate}[a]
        \item Let \(\tau\) be an arbitrary isolated vertex of \(L(T)\).  It is readily verified that \(\tau = b_l\) for some \(1 \le l \le t-1\).  Furthermore, \(b_{l-1}\) and \(b_{l+1}\) lie in the same connected component of \(G \setminus T\), which we denote by \(G'\).  In particular, \(b_l - 1\) and \(b_l + 1\) belong to \(G'\).

            Thus, for every \(1 \le k_1 < k_2 \le m\), the binomial \([k_1,k_2|b_l - 1,b_l + 1]\) is contained in the prime ideal \(\mathfrak{P}_T(K_m,G)\).  Since \((\calJ_{K_m,G} : f) = \mathfrak{P}_T(K_m,G)\), this implies \(f\) is in the colon ideal
            \begin{align*}
                (\calJ_{K_m,G} : [k_1,k_2|b_l - 1,b_l + 1]) = (x_{k,b_l} \mid k \in [m]) + \calJ_{K_m,G \setminus b_l},
            \end{align*}
            by \Cref{lem:colon_edge_short}.  Therefore,
            \[
                \ini(f) \in (x_{k,b_l} \mid k \in [m]) + \ini(\calJ_{K_m,G \setminus b_l}).
            \]
            Since \(\ini(\calJ_{K_m,G \setminus b_l}) \subseteq \ini(\calJ_{K_m,G})\) but \(\ini(f) \notin \ini(\calJ_{K_m,G})\), there exists some \(\gamma_l \in [m]\) such that \(x_{\gamma_l,b_l}\) divides \(\ini(f)\).

        \item Let \(\tau\) be an arbitrary element of \(T\); then \(\tau = b_{j_i}\) for some \(i \in [s]\).
For every \(k \in [m]\), we have \(x_{k,b_{j_i}} \in \mathfrak{P}_T(K_m,G)\), and thus \(f \in (\calJ_{K_m,G} : x_{k,b_{j_i}})\).
As shown in \Cref{thm:colon_x}, this colon ideal is equal to \(\calJ_{K_m,G_{b_{j_i}}}\).
Recall that \(G_{b_{j_i}}\) is a Cohen--Macaulay closed graph with \(E(G_{b_{j_i}}) = E(G) \cup \{\{u,v\} \mid u,v \in N_G(b_{j_i})\}\).
Since \(\ini(f) \notin \ini(\calJ_{K_m,G})\), we conclude there exists a monomial \(x_{k_i',w_i'}x_{k_i'',w_i''}\) dividing \(\ini(f)\) such that \(1 \le k_i' < k_i'' \le m\) and \(b_{j_i-1} \le w_i' < b_{j_i} < w_i'' \le b_{j_i+1}\).
Notably, neither \(w_i'\) nor \(w_i''\) is an isolated vertex in \(L(T)\).
Furthermore, we have \(b_{j_i-1} \le \alpha_i < b_{j_i} < \beta_i \le b_{j_i+1}\).
    \end{enumerate}

Therefore, the squarefree monomial
    \[
        M \coloneqq \operatorname{lcm}\bigl(\{x_{\gamma_l,b_l} \mid b_l \text{ is an isolated vertex of } L(T)\}
        \cup\{x_{k_i',w_i'}x_{k_i'',w_i''} \mid b_{j_i}\in T\}\bigr)
    \]
    divides \(\ini(f)\).  We next consider the ``flattened'' monomial
    \[
        \widetilde{M} \coloneqq \operatorname{lcm}\bigl(\{x_{\gamma_l,b_l} \mid b_l \text{ is an isolated vertex of } L(T)\}
        \cup\{x_{k_i',\alpha_i}x_{k_i'',\beta_i} \mid b_{j_i}\in T\}\bigr).
    \]
    It is straightforward to see that \(\deg(\widetilde{M}) \le \deg(M)\).  Furthermore, if the variable \(x_{i,v}\) divides \(\widetilde{M}\), then \(v \in V(L(T))\).

    Consider an equivalence relation (partition) \(\calP\) on \(E(L(T))\) generated by the following rule: if \(x_{k_i'',\beta_i} = x_{k_{i+1}',\alpha_{i+1}}\), then \(\{\alpha_i,\beta_i\} \sim \{\alpha_{i+1},\beta_{i+1}\}\).  Suppose that there are \(q\) equivalence classes, and the vertex sets of the respective equivalence classes are
    \[
        \{c_{i_p,1} < c_{i_p,2} < \dots < c_{i_p,\ell_p}\}
    \]
    for \(p \in [q]\) and \(2 \le \ell_p\).

    For each \(p \in [q]\), suppose that \(c_{i_p,r} = \alpha_{j_r}\) for \(r \in [\ell_p-1]\).  Observe that
    \[
        1 \le k_{j_1}' < k_{j_1}'' = k_{j_2}' < \cdots < k_{j_{\ell_p-2}}'' = k_{j_{\ell_p-1}}' < k_{j_{\ell_p-1}}'' \le m.
    \]
    Hence \(\ell_p \le m\), so in particular \(\calP\) is an \(m\)-compatible partition of \(E(L(T))\).  Furthermore, the squarefree monomial
    \[
        \operatorname{lcm}\bigl\{x_{k_j',\alpha_j}x_{k_j'',\beta_j} \mid j = j_r \text{ for } r \in [\ell_p-1]\bigr\}
    \]
    equals
    \[
        x_{k_{j_1}',\alpha_{j_1}} x_{k_{j_2}',\alpha_{j_2}} \cdots x_{k_{j_{\ell_p-1}}',\alpha_{j_{\ell_p-1}}}
        x_{k_{j_{\ell_p-1}}'',\beta_{j_{\ell_p-1}}}.
    \]
    We denote this monomial by \(\widetilde{f}_p\).

    It is easy to see that
    \begin{align*}
        \widetilde{M} =
        \left(\prod_{\substack{b_l \text{ is an isolated}\\\text{vertex of } L(T)}}
        x_{\gamma_l,b_l}\right)
        \cdot \prod_{p=1}^q \widetilde{f}_p.
    \end{align*}
    Since \(\deg(M) \ge \deg(\widetilde{M}) = \deg(f_{\calP})\) and \(M\) divides \(\ini(f)\), we obtain \(\deg(f) \ge \deg(f_{\calP})\), as desired.
\end{proof}

\begin{Remark}
    Note that the graph \(L(T)\) in \Cref{construct:L_T_non_CM} is defined specifically for non-empty cut sets of \(G\). Indeed, if \(T\) is the empty set, we may take \(L(T)\) to be the graph consisting of \(\gamma_c^{*}(G)\) isolated vertices.  Thus the proof of \Cref{prop:local_v_number_path} still carries over without difficulty, which further confirms that the local \(\v\)-number equals \(\gamma_c^{*}(G)\), as predicted in \Cref{lem:minimal_CDS}.
\end{Remark}

\subsection{Realization of \(\v(\calJ_{K_m,G})\)}

It follows from \Cref{prop:local_v_number_path} that there exists an \(m\)-compatible partition \(\calP\) of \(E(L(T))\) such that \(T\in \calC(G)\) and \(\deg(f_{\calP})\) attains the desired minimum value, that is,
\[
    \v(\calJ_{K_m,G}) = \deg(f_{\calP}) = \min\{\v_W(\calJ_{K_m,G}) \mid W\in \calC(G)\}.
\]
We now make the following observations concerning this optimal cut set \(T\) and the optimal partition \(\calP\).

\begin{Observations}
    \label{obs:optimal_T_P}
    \begin{enumerate}[a]
        \item \label{item_no_1_t+1}
            If \(1 = b_0 \in V(L(T))\), then \(b_1 = b_{j_1} \in T\) and \(b_0\) is not an isolated vertex.  In other words, \(b_0\) is the left endpoint of some slice from the optimal partition \(\calP\).  We remove the edge \(\{b_0 = \alpha_1, \beta_1\}\) from \(E(L(T))\) but retain the vertex \(b_0\) in \(V(L(T))\).  This yields a graph \(L'(T')\) that is closely related to \(L(T')\) for the cut set \(T' \coloneqq T \setminus \{b_1\} \in \calC(G)\).  For the resulting partition \(\calP'\) on \(E(L(T'))\), we obtain the new polynomial \(f_{\calP'}\).

            \begin{enumerate}[i]
                \item If  \(b_2 = b_{j_2} \in T\), then \(b_0\) is an isolated vertex of \(L'(T')\).  Removing \(b_0\) yields a graph isomorphic to \(L(T')\).  In this case, \(\deg(f_{\calP}) > \deg(f_{\calP'})\), which contradicts the minimality of \(\calP\).

                \item Otherwise, \(b_2 \notin T\) and the single edge \(\{b_0,b_2\}\) is a slice of length \(1\) in \(\calP\).  It is easy to see that \(L'(T')\) is isomorphic to \(L(T')\), with the isolated vertex \(b_0\) replaced by \(b_1\).  In this case, \(\deg(f_{\calP}) \ge \deg(f_{\calP'})\).  The strict inequality \(\deg(f_{\calP}) > \deg(f_{\calP'})\) holds if and only if \(b_2\) is also the endpoint of another adjacent slice in \(\calP\).  
            \end{enumerate}

            Thus, for simplicity, we may assume \(b_0 \notin V(L(T))\).  By symmetry, we may also assume \(n = b_t \notin V(L(T))\).

        \item \label{item_no_bij}
            Recall that \(\widetilde{C}(G) \coloneqq \{b_1, \dots, b_{t-1}\}\) denotes the set of cut vertices of \(G\).  We currently have \(\widetilde{C}(G) \subseteq V(L(T)) \sqcup T\).  If this containment is strict, then by the assumption in Item \ref{item_no_1_t+1} above, there exists some \(\beta_i = \alpha_{i+1} \notin \widetilde{C}(G)\).  Let \(i\) be the minimal index for which this holds.  Note that \(\beta_i = \alpha_{i+1}\) is adjacent to both \(\alpha_i\) and \(\beta_{i+1}\).  We remove the edge \(\{\alpha_i, \beta_i\}\) from the graph \(L(T)\).  As in Item \ref{item_no_1_t+1} above, the resulting graph is isomorphic to \(L(T')\) for the cut set \(T' \coloneqq T \setminus \{b_{j_i}\} \in \calC(G)\).  Let \(\calP'\) be the induced partition on \(E(L(T'))\); one can verify that \(\deg(f_{\calP}) \ge \deg(f_{\calP'})\).

            Thus, for simplicity, we may assume \(\widetilde{C}(G) = V(L(T)) \sqcup T\).
As a result, \(j_i + 2 \le j_{i+1}\) for all \(i \in [s-1]\).

        \item \label{item_few_isolated_vertices}
            Without loss of generality, we may assume all isolated vertices lie at the right (larger) end of \(\widetilde{C}(G)\). We claim \(L(T)\) contains at most two isolated vertices. To see this, suppose for contradiction that \(L(T)\) has at least three isolated vertices. We may take such vertices to be \(b_{t-3}, b_{t-2}, b_{t-1}\). Adding \(b_{t-2}\) to \(T\) gives a new cut set \(T'\) satisfying
            \[
                E(L(T')) = E(L(T)) \cup \{\{b_{t-3}, b_{t-1}\}\}.
            \]
            For the corresponding partition \(\calP'\) of \(E(L(T'))\) where the single edge \(\{b_{t-3}, b_{t-1}\}\) is a slice, we have \(\deg(f_{\calP'}) = \deg(f_{\calP}) - 1\), contradicting the minimality of \(\calP\).

            Thus, we may assume \(L(T)\) contains at most two isolated vertices.

        \item \label{item_slices_disjoint}
            Suppose that there exist two adjacent slices with lengths $(n_1-1)$ and $(n_2-1)$ respectively, where \(2\le n_1,n_2\le m\), and that they are connected by a common vertex $b_{j_{v}}$. Due to the assumption in the previous Item \ref{item_no_bij}, the leftmost edge of the right (second) slice must be $\{b_{j_{v}},b_{j_{v+2}}\}$, with $b_{j_{v+1}}\in T$. Removing the edge $\{b_{j_{v}},b_{j_{v+2}}\}$ from $L(T)$ and supplementing it with an isolated vertex $b_{j_{v+1}}$ yields $L(T')$ for the cut set $T'\coloneqq T\setminus\{b_{j_{v+1}}\}\in \calC(G)$. Let \(\calP'\) be the resulting partition on \(E(L(T'))\). Notice that \(\deg(f_{\calP})\ge \deg(f_{\calP'})\). 

            Therefore, for simplicity, we may assume that all the slices resulting from the partition $\calP$ are vertex-disjoint.

        \item Suppose there exist two adjacent but non-connected slices of lengths \(n_1 - 1\) and \(n_2 - 1\), respectively, with \(2 \le n_1 < m\) and \(2 \le n_2 \le m\). We move the left end-edge of the right slice to the left slice, creating two new slices of lengths \((n_1 + 1) - 1\) and \((n_2 - 1) - 1\), respectively. If \(n_2 = 2\), we regard the resulting right slice, an isolated vertex, as a degenerate slice. After suitable relabeling, the resulting graph corresponds to another \(m\)-compatible partition \(\calP'\) of \(E(L(T'))\) for some cut set \(T' \in \calC(G)\). Note that \(\deg(f_{\calP}) = \deg(f_{\calP'})\). 

            Thus, by the assumptions in Items \ref{item_few_isolated_vertices} and \ref{item_slices_disjoint} and the aforementioned movement, we may assume for simplicity that all slices have length \(m - 1\), except the rightmost slice.

        \item Suppose the vertex set of the rightmost slice of \(L(T)\) is \(\{c_{i_q,1} < c_{i_q,2} < \dots < c_{i_q,\ell_q}\}\) with \(2 \le \ell_q \le m\). We claim that \(\ell_q = m\) if \(L(T)\) has two isolated vertices.

            To verify this, suppose \(c_{i_q,\ell_q} = b_{j_v}\). It follows that the two isolated vertices are \(b_{j_v+1}\) and \(b_{j_v+2}\), which are indeed \(b_{t-2}\) and \(b_{t-1}\), respectively. Now assume for contradiction that \(\ell_q < m\). We may replace \(T\) with \(T' \coloneqq T \cup \{b_{j_v+1}\} \in \calC(G)\). For such \(T'\), we have
            \[
                V(L(T')) = V(L(T)) \setminus \{b_{j_v+1}\}
                \quad\text{and}\quad
                E(L(T')) = E(L(T)) \cup \bigl\{\{b_{j_v},b_{j_v+2}\}\bigr\}.
            \]
            Extending the rightmost slice of \(\calP\) by the newly added edge yields an \(m\)-compatible partition \(\calP'\) of \(E(L(T'))\). Since \(\deg(f_{\calP'}) < \deg(f_{\calP})\), this contradicts the minimality of \(\calP\).

            We therefore conclude that \(\ell_p = m\) whenever \(L(T)\) contains two isolated vertices.
    \end{enumerate} 
\end{Observations}

To summarize, for the optimal cut set \(T\) and the optimal partition \(\calP\) of \(E(L(T))\) satisfying \(\v(\calJ_{K_m,G}) = \deg(f_{\calP})\), we may impose the following assumptions:
\begin{itemize}
    \item The vertex set \(V(L(T))\) of \(L(T)\) satisfies \(V(L(T)) \sqcup T = \widetilde{C}(G)\);
    \item The slices are pairwise vertex-disjoint;
    \item Every slice has length \(m - 1\), except possibly the rightmost slice;
    \item The number of isolated vertices is at most two, and these vertices lie near the right end of the graph;
    \item If there are two isolated vertices, the rightmost slice also has length \(m - 1\).
\end{itemize}
Under these assumptions, both the optimal cut set \(T\) and the optimal partition \(\calP\) are uniquely determined.

\begin{Construction}
    \label{construct:T}
    We describe the construction of the optimal cut set \(T\) and the optimal partition \(\calP\) as follows. Let \(\widetilde{C}(G) = \{b_1, \dots, b_{t-1}\}\) denote the base vertex set.

    We may write \(t-1 = \left\lfloor\frac{t-1}{2m-1}\right\rfloor \cdot (2m-1) + A\) for some integer \(A\) with \(0 \le A < 2m-1\). Using this decomposition, for each \(p \in \left[\left\lfloor\frac{t-1}{2m-1}\right\rfloor\right]\) and each \(j \in [2m-1]\), set \(d_{p,j} \coloneqq b_{(p-1)(2m-1) + j}\). This gives \(\left\lfloor\frac{t-1}{2m-1}\right\rfloor\) slices of length \((2m-1) - 1\), whose vertex sets are respectively
    \[
        \left\{d_{p,1} < \fbox{$d_{p,2}$} < d_{p,3} < \fbox{$d_{p,4}$} < d_{p,5} < \cdots < \fbox{$d_{p,2m-2}$} < d_{p,2m-1}\right\}
    \]
    for \(p = 1, 2, \dots, \left\lfloor\frac{t-1}{2m-1}\right\rfloor\).

    When \(A > 0\), we have \(A + 1 = \left\lfloor\frac{A+1}{2}\right\rfloor \cdot 2 + B\) for some \(B \in \{0, 1\}\). Let \(K = \left\lfloor\frac{A+1}{2}\right\rfloor \cdot 2 - 1\), and for each \(j \in [K]\), set \(e_j \coloneqq b_{\left\lfloor\frac{t-1}{2m-1}\right\rfloor \cdot (2m-1) + j}\). Corresponding to this, there is a slice of length \(K - 1\) with vertex set
    \[
        \{e_1 < \fbox{$e_2$} < e_3 < \fbox{$e_4$} < e_5 < \cdots < \fbox{$e_{K-1}$} < e_K\}.
    \]
    If \(K = 1\), we regard this as a degenerate slice, i.e., an isolated vertex. There may be \(B\) extra isolated vertices remaining.

    If \(A = 0\), we set \(B = 0\) as well.

    For this construction, the optimal cut set \(T\) is exactly the set of all boxed vertices. Removing these vertices from the above construction yields the optimal partition \(\calP\) of \(E(L(T))\).
\end{Construction}

\begin{Example}
    Let $G$ be the Cohen--Macaulay closed graph in \Cref{exam1}. If $m=3$, the optimal cut set constructed in \Cref{construct:T} is \(T=\{6, 9, 15,19, 24\}\). The associated graph $L(T)$ is depicted in \Cref{fig3}, which has no isolated vertices.
    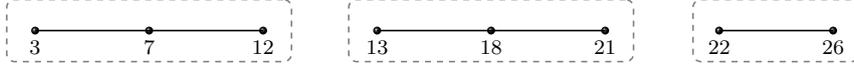
\begin{figure}[tbhp] 
        \scalebox{0.75}{
            \begin{tikzpicture}[thick, scale=1, every node/.style={scale=1.3}]
                \shade [shading=ball, ball color=black] (0,0) circle (.07) node [below] {\scriptsize$3$};
                \shade [shading=ball, ball color=black] (2,0) circle (.07) node [below] {\scriptsize$7$};
                \shade [shading=ball, ball color=black] (4,0) circle (.07) node [below] {\scriptsize$12$};
                \shade [shading=ball, ball color=black] (6,0) circle (.07) node [below] {\scriptsize$13$};
                \shade [shading=ball, ball color=black] (8,0) circle (.07) node [below] {\scriptsize$18$};
                \shade [shading=ball, ball color=black] (10,0) circle (.07) node [below] {\scriptsize$21$};
                \shade [shading=ball, ball color=black] (12,0) circle (.07) node [below] {\scriptsize$22$};
                \shade [shading=ball, ball color=black] (14,0) circle (.07) node [below] {\scriptsize$26$};

                \draw (0,0)--(4,0);
                \draw (6,0)--(10,0);
                \draw (12,0)--(14,0);

                \node [draw, rounded corners, fit={(0.5,0) (3.5,0)}, inner sep=12pt, dashed, gray] {};
                \node [draw, rounded corners, fit={(6.5,0) (9.5,0)}, inner sep=12pt, dashed, gray] {};
                \node [draw, rounded corners, fit={(12.3,0) (13.7,0)}, inner sep=12pt, dashed, gray] {};
            \end{tikzpicture}
        }
        \caption{The associated graph $L(T)$ for the optimal cut set $T$}
        \label{fig3}
    \end{figure} 
\end{Example}

We are now in a position to state the main result of this section.

\begin{Theorem}
    \label{thm:v_number_CM_closed}
    If $G$ is a connected Cohen--Macaulay closed graph with $t$ maximal cliques, then 
    \[
        \v(\calJ_{K_m,G})= D_{m,t}\coloneqq \floor{\frac{t-1}{2m-1}}\cdot m+ \floor{\frac{A+1}{2}} + B.
    \]
\end{Theorem}
\begin{proof}
    Let \(\calP\) be the optimal partition of \(E(L(T))\) constructed in \Cref{construct:T}. The polynomial \(f_{\calP}\) associated with this partition has degree \(D_{m,t}\). By the characterization of the local \(\v\)-number given in \Cref{thm:local_v_number_non_CM_closed}, together with the relevant observations in \Cref{obs:optimal_T_P}, it follows that \(\v(\calJ_{K_m,G}) = D_{m,t}\).
\end{proof}

\begin{Corollary}
    Let \(P_n\) denote a path  with \(n\) vertices. Then \(\v(\calJ_{K_m,P_n}) = D_{m,n-1}\), where \(D_{m,n-1}\) is defined as in \Cref{thm:v_number_CM_closed} with \(t = n-1\).
\end{Corollary}
\begin{proof}
    It suffices to observe that the path   \(P_n\) is a connected Cohen--Macaulay closed graph with exactly \(n-1\) maximal cliques. By \Cref{thm:v_number_CM_closed}, the assertion follows immediately.
\end{proof}

\begin{Corollary}
    \label{CMclsoedgraph}
    Let \(G\) be a Cohen--Macaulay closed graph with \(t\) maximal cliques. Then \(\v(\calJ_{K_2,G}) = \ceil{\frac{2(t-1)}{3}}\).
\end{Corollary}
\begin{proof}
    It suffices to verify that \(D_{2,t} = \ceil{\frac{2(t-1)}{3}}\), where \(D_{2,t}\) is the expression defined in \Cref{thm:v_number_CM_closed} with \(m = 2\). This equality is elementary to check.
\end{proof}

\section{$\v$-number of powers of binomial edge ideals of  connected  Cohen--Macaulay closed graphs}
\label{sec:power}

In this section, we keep \(G\) as a connected Cohen--Macaulay closed graph. Our aim is to determine the \(\v\)-number of powers of the binomial edge ideal \(\calJ_{K_2,G}\).

Accordingly, we assume there exist integers \(1 = b_0 < b_1 < \cdots < b_{t-1} < b_t = n\) such that the maximal cliques of \(G\) are given by \(F_i = [b_{i-1}, b_i]\) for \(i \in [t]\). Let \(P\) be the spine of \(G\); that is, \(P\) is the path graph with vertex set \(V(P) = \{b_0, b_1, \dots, b_t\}\).

We first establish two basic lemmas concerning initial ideals and associated primes of powers of \(\calJ_{K_2,G}\).
These results will be used repeatedly throughout the section.

\begin{Lemma}
    \label{lem:power_initial}
    For every positive integer \(k\), we have  $\ini(\calJ_{K_2, G}^k)=(\ini(\calJ_{K_2, G}))^k$ and \(\Ass(\calJ_{K_2, G}^k)=\Ass(\calJ_{K_2, G})\).
\end{Lemma}
\begin{proof}
    It follows from \cite[Corollary 3.4]{MR4143239} that \(\calJ_{K_2, G}^k = \calJ_{K_2, G}^{(k)}\), where \(\calJ_{K_2, G}^{(k)}\) denotes the \(k\)-th symbolic power of \(\calJ_{K_2, G}\). Since \(\calJ_{K_2,G}\) is radical, we have \(\Ass(\calJ_{K_2, G}^k) = \Ass(\calJ_{K_2, G}^{(k)}) = \Ass(\calJ_{K_2, G})\). Moreover, the proof of \cite[Corollary 3.4]{MR4143239} implies that \cite[Lemma 3.1]{MR4143239} is applicable. In particular, the proof of \cite[Lemma 3.1]{MR4143239} yields \(\ini(\calJ_{K_2, G}^k) = (\ini(\calJ_{K_2, G}))^k\).
\end{proof}

In what follows, for an edge \(\varepsilon=\{k,l\}\in E(G)\) with \(1\le k<l\le n\), we set \(f_\varepsilon\coloneqq x_{1,k}x_{2,l}-x_{1,l}x_{2,k}\in S=\KK[x_{1,1},\dots,x_{2,n}]\).

We next prove a key colon ideal identity that controls the behavior of powers with respect to the extremal edge \(\{1,2\}\).

\begin{Lemma}
    \label{lem:power_colon_g}
    For every positive integer \(k\), we have \((\calJ_{K_2, G}^k:g) = \calJ_{K_2, G}^{k-1}\), where \(g\coloneqq f_{\{1,2\}}\).
\end{Lemma}
\begin{proof}
    Since \(\{1,2\} \in E(G)\), the case \(k = 1\) is trivial; we therefore assume \(k \ge 2\). It is clear that \(\calJ_{K_2, G}^{k-1} \subseteq (\calJ_{K_2, G}^k : g)\). Thus, it suffices to show \(\ini(\calJ_{K_2, G}^k : g) \subseteq \ini(\calJ_{K_2, G}^{k-1})\).

    To establish this, let \(f \in (\calJ_{K_2, G}^k : g)\). By \Cref{lem:power_initial}, we have \(\ini(f)\ini(g) \in \ini(\calJ_{K_2, G}^k) = (\ini(\calJ_{K_2, G}))^k\). Note that \(\ini(g) = x_{1,1}x_{2,2}\), and \(x_{1,1}x_{2,2}\) is the unique minimal monomial generator of \(\ini(\calJ_{K_2, G})\) divisible by \(x_{2,2}\). It follows immediately that \(((\ini(\calJ_{K_2, G}))^k : \ini(g)) = (\ini(\calJ_{K_2, G}))^{k-1}\). This implies \(\ini(f) \in (\ini(\calJ_{K_2, G}))^{k-1}\), and hence \(\ini(\calJ_{K_2, G}^k : g) \subseteq \ini(\calJ_{K_2, G}^{k-1})\), as desired.
\end{proof}

To understand the local \(\v\)-numbers of powers, we first analyze the case where \(G\) is a path graph, which serves as the building block for all connected Cohen--Macaulay closed graphs.

\begin{Lemma}
    \label{lem:initial_works_for_path}
    Let \(G = P\) be a path graph and let \(T\) be a cut set of \(G\). For any monomial
    \[
        h \in (\ini(\calJ_{K_2,G}) : \ini(\frakP_T(K_2,G))) \setminus \ini(\calJ_{K_2,G}),
    \]
    we have \(\deg(h) \ge \v_{T}(\calJ_{K_2,G})\).
\end{Lemma}
\begin{proof}
    Recall that for the cut set \(T\), we may construct the graph \(L(T)\). Since \(m = 2\), there exists a unique \(2\)-compatible partition of \(E(L(T))\), which is trivial in the sense that each equivalence class contains exactly one edge. Let \(\calP_0\) denote this unique \(2\)-compatible partition. By abuse of notation, we write \(f_T\) for the polynomial \(f_{\calP_0}\). It then follows from \Cref{thm:local_v_number_non_CM_closed} that \(\deg(f_T) = \v_{T}(\calJ_{K_2,G})\).

    For the given monomial \(h\), we make the following observations:
    \begin{enumerate}[a]
        \item For each vertex \(v\in T\), since \(x_{1,v},x_{2,v}\in \ini(\frakP_T(K_2,G))\) and
            \[
                h\,\ini(\frakP_T(K_2,G))\subseteq \ini(\calJ_{K_2, G})=(x_{1,i}x_{2,i+1}\mid i\in [n-1]),
            \]
            we conclude that both \(x_{2,v+1}\) and \(x_{1,v-1}\) divide \(h\).  In other words, \(h\) is divisible by \(x_{1,v-1}x_{2,v+1}\).  Note that \(\{v-1,v+1\}\) is an edge of \(L(T)\). Neither \(v-1\) nor \(v+1\) is an isolated vertex of \(L(T)\).

        \item For each isolated vertex \(u\) of \(L(T)\), the vertices \(u-1\) and \(u+1\) lie in the same connected component of \(G\setminus T\).  Hence \(x_{1,u-1}x_{2,u+1}\in \ini(\frakP_T(K_2,G))\).  Since \(h\,\ini(\frakP_T(K_2,G))\subseteq \ini(\calJ_{K_2, G})\), it follows that \(h\) is divisible by either \(x_{1,u}\) or \(x_{2,u}\).  For simplicity, we denote this element by \(z_u\).
    \end{enumerate}

    From the above reasoning, we deduce that \(h\) is divisible by
    \begin{equation}
        h'=\lcm\bigl(\{\uline{x_{1,v-1}x_{2,v+1}}\mid v\in T\}\cup\{\uline{z_u}\mid \text{$u$ is an isolated vertex of $L(T)$}\}\bigr).
        \label{eqn:h'}
    \end{equation}
    It is straightforward to check that \(h'\) is exactly the product of the underlined monomials in \eqref{eqn:h'} and that \(\deg(h')=\deg(f_T)\).  Accordingly, we obtain \(\deg(h)\ge \v_{T}(\calJ_{K_2, G})\).
\end{proof}

Using the preceding lemma, we can now determine the local \(\v\)-numbers for powers of binomial edge ideals over path graphs.

\begin{Proposition}
    \label{powers1}
    Let \(G = P\) be a path graph on vertex set \([n]\) with \(n \ge 3\), and let \(T\) be a cut set of \(G\). Then for any integer \(k \ge 1\),
\[
\v_{T}(\calJ_{K_2,G}^k) = \v_{T}(\calJ_{K_2,G}) + 2(k - 1).
\]
\end{Proposition}
\begin{proof}
    Let \(f\) be a homogeneous polynomial in \(S\) such that \((\calJ_{K_2, G}:f) = \frakP_T(K_2,G)\) and \(\deg(f) = \v_T(\calJ_{K_2, G})\).  By \Cref{lem:power_colon_g}, it follows that \((\calJ_{K_2, G}^k:g^{k-1}f) = \frakP_T(K_2,G)\), where \(g = f_{\{1,2\}}\). Since \(\Ass(\calJ_{K_2, G}^k) = \Ass(\calJ_{K_2, G})\), this implies \(\v_{T}(\calJ_{K_2, G}^k) \le \deg(g^{k-1}f) = \v_{T}(\calJ_{K_2, G}) + 2(k-1)\). It remains to show that \(\v_{T}(\calJ_{K_2, G}^k) \ge \v_{T}(\calJ_{K_2, G}) + 2(k-1)\), which we establish below.

    Let \(h\) be a homogeneous polynomial such that \((\calJ_{K_2, G}^k:h) = \frakP_T(K_2,G)\).  We will show that
    \begin{equation}
        \deg(h) \ge \v_{T}(\calJ_{K_2, G}) + 2(k-1).
        \label{eqn:deg_h}
    \end{equation}
    By \Cref{lem:change_initial}, we may assume \(\ini(h) \notin \ini(\calJ_{K_2, G}^k)\). Note that \(\ini(h)\ini(\frakP_T(K_2,G)) \subseteq \ini(\calJ_{K_2, G}^k) = (\ini(\calJ_{K_2, G}))^k\), where the last equality holds by \Cref{lem:power_initial}.

    If \(T \ne \emptyset\), take an arbitrary \(v \in T\). Since \(x_{1,v} \in \ini(\frakP_T(K_2,G))\) and \(\ini(h) \notin \ini(\calJ_{K_2, G}^k)\), it follows that \(\ini(h) = x_{2,v+1}\ini(f_{e_1})\cdots \ini(f_{e_{k-1}})h'\) for some edges \(e_1,\dots,e_{k-1} \in E(G)\) and some monomial \(h' \in S\). Since \(\ini(\calJ_{K_2, G})\) is a complete intersection, 
    \[
        \left((\ini(\calJ_{K_2, G}))^k: \ini(f_{e_1})\cdots \ini(f_{e_{k-1}})\right) = \ini(\calJ_{K_2, G})
    \]
    by \cite[Exercise 1.1.15]{MR1251956}. This implies
    \[
        x_{2,v+1}h' \in \left((\ini(\calJ_{K_2, G}))^k: \ini(f_{e_1})\cdots \ini(f_{e_{k-1}})\ini(\frakP_T(K_2,G))\right) = \left(\ini(\calJ_{K_2, G}):\ini(\frakP_T(K_2,G))\right).
    \]
    As a consequence, \Cref{lem:initial_works_for_path} yields 
    \[
        \deg(h) - 2(k-1) = \deg(x_{2,v+1}h') \ge \v_{T}(\calJ_{K_2, G}).
    \]
    Thus, the inequality \eqref{eqn:deg_h} holds in this case.

    If \(T = \emptyset\), then \(\v_{\emptyset}(\calJ_{K_2, G}) = n-2\) by Corollary \ref{cor:v_empty}. For each \(v \in [n-2]\), since \(x_{1,v}x_{2,v+2} \in \ini(\frakP_\emptyset(K_2,G))\) and \(\ini(h) \notin \ini(\calJ_{K_2, G}^k)\), either 
    \[
        x_{2,v+2}\ini(h) =x_{2,v+1}\ini(f_{e_1})\cdots \ini(f_{e_{k-1}})h'
    \]
    or 
    \[
        x_{1,v}\ini(h) = x_{1,v+1}\ini(f_{e_1})\cdots \ini(f_{e_{k-1}})h'
    \]
    for some edges \(e_1,\dots,e_{k-1} \in E(G)\) and some monomial \(h' \in S\).

    \begin{enumerate}[i]
        \item Suppose \(x_{2,v+2}\ini(h) =x_{2,v+1}\ini(f_{e_1})\cdots \ini(f_{e_{k-1}})h'\). If \(x_{2,v+2}\) divides \(h'\), then \(\ini(h)\) is divisible by \(\ini(f_{e_1})\cdots \ini(f_{e_{k-1}})\). Using the complete intersection property, the remainder of the proof proceeds similarly to the \(T\ne \emptyset\) case. Otherwise, we may assume \(x_{2,v+2}\) divides \(\ini(f_{e_{k-1}})\). Since \(\ini(f_{e_{k-1}})\) must equal \(x_{1,v+1}x_{2.v+2}\), this implies \(x_{1,v+1}x_{2,v+1}\ini(f_{e_1})\cdots \ini(f_{e_{k-2}})\) divides \(\ini(h)\).

        \item Suppose \(x_{1,v}\ini(h) = x_{1,v+1}\ini(f_{e_1})\cdots \ini(f_{e_{k-1}})h'\). This case is analogous to the previous one.
    \end{enumerate}
    Thus, we only need to consider the case where, for each \(v \in [n-2]\), there exists a minimal monomial generator \(F_{v}\) of \(\ini(\calJ_{K_2, G}^{k-2})\) such that \(F_{v} x_{1,v+1}x_{2,v+1}\) divides \(\ini(h)\).

    \begin{enumerate}[a]
        \item Suppose there exist \(v,v' \in [n-2]\) such that \(F_{v} \ne F_{v'}\).  Since both \(F_v\) and \(F_{v'}\) divide \(\ini(h)\), Then \(\lcm(F_{v},F_{v'})\) divides \(\ini(h)\), and \(\deg(\lcm(F_{v},F_{v'})) \ge 2(k-1)\). This implies \(\ini(h)\) is divisible by \(\ini(f_{e_1})\cdots \ini(f_{e_{k-1}})\) for some edges \(e_1,\dots,e_{k-1} \in E(G)\). Using the complete intersection property, the remainder of the proof proceeds similarly to the \(T\ne \emptyset\) case.

        \item Otherwise, \(F_1 = \cdots = F_{n-2}\). Let \(F\) denote this common monomial. Then \(\ini(h)\) is divisible by \(Fx_{1,2}x_{2,2}x_{1,3}x_{2,3}\cdots x_{1,n-1}x_{2,n-1}\).

            When \(n \ge 4\), this implies
            \[
                \deg(h) \ge 2(k-2) + 2(n-2) \ge 2(k-1) + (n-2) = 2(k-1) + \v_{\emptyset}(\calJ_{K_2, G}),
            \]
            as desired.

            When \(n = 3\), using the complete intersection property, we have 
            \[
                \ini(h)/F \in (\ini(\calJ_{K_2, G}^k):F\ini(\frakP_\emptyset(K_2,G)))= (\ini(\calJ_{K_2, G}^2):\ini(\frakP_\emptyset(K_2,G))).
            \]
            It is straightforward to verify that
            \begin{align*}
                (\ini(\calJ_{K_2, G}^2):\ini(\frakP_\emptyset(K_2,G)))&=(x_{1,1}x_{2,2},x_{1,2}x_{2,3})^2:(x_{1,1}x_{2,2},x_{1,2}x_{2,3},x_{1,1}x_{2,3}) \\
                &= (x_{1,2}x_{2,2}x_{2,3}, x_{1,2}^2x_{2,3}, x_{1,1}x_{2,2}^2, x_{1,1}x_{1,2}x_{2,2}).
            \end{align*}
            Thus, \(\deg(h) - 2(k-2) \ge 3\), i.e., \(\deg(h) \ge 2(k-1) + \v_{\emptyset}(\calJ_{K_2, G})\), as desired. \qedhere
    \end{enumerate}
\end{proof}

Having established the formula for path graphs, we now return to the general case of connected Cohen--Macaulay closed graphs.
The spine subgraph allows us to reduce the general case to the path graph case.

\begin{Theorem} 
    \label{powers2}
    Let \(G\) be a connected Cohen--Macaulay closed graph. For every positive integer \(k\), we have
    \[
        \v(\calJ_{K_2,G}^k) = \v(\calJ_{K_2,G})+2(k - 1).
    \]
\end{Theorem}

\begin{proof}
    Let \(f\) be a homogeneous polynomial in \(S\) such that \((\calJ_{K_2, G}:f)=\frakP_T(K_2,G)\) for some \(T\in \calC(G)\) and \(\deg(f)=\v(\calJ_{K_2, G})\).  By \Cref{lem:power_colon_g}, we have \((\calJ_{K_2, G}^k:g^{k-1}f)=\frakP_T(K_2,G)\) for \(g\coloneqq f_{\{1,2\}}\).
    Since \(\Ass(\calJ_{K_2, G}^k)=\Ass(\calJ_{K_2, G})\), it follows that \(\v_{T}(\calJ_{K_2, G}^k)\le \deg(g^{k-1}f)=\v(\calJ_{K_2, G})+2(k-1)\).  In particular, \(\v(\calJ_{K_2, G}^k)\le \v(\calJ_{K_2, G})+2(k-1)\).

    It remains to show that \(\v(\calJ_{K_2, G}^k)\ge \v(\calJ_{K_2, G})+2(k-1)\).  To this end, let \(\varphi\) be the \(\KK\)-algebra homomorphism from \(S=\KK[x_{1,1},\dots,x_{1,n},x_{2,1},\dots,x_{2,n}]\) onto its subring
    \[
        R=\KK[x_{1,b_0},x_{1,b_1},\dots,x_{1,b_t},x_{2,b_0},x_{2,b_1},\dots,x_{2,b_t}],
    \]
    defined by
    \begin{itemize}
        \item the restriction of \(\varphi\) to \(R\) is the identity map;
        \item if \(b_j<v<b_{j+1}\), then \(\varphi(x_{1,v})=x_{1,b_j}\) and \(\varphi(x_{2,v})=x_{2,b_{j+1}}\).
    \end{itemize}
    It is straightforward to check that \(\varphi(\calJ_{K_2, G})=\calJ_{K_2, P}\), where \(P\) is the spine of \(G\) with vertex set \(V(P)=\{b_0,b_1,\dots,b_t\}\).  Note that \(\calC(P)\subseteq \calC(G) \subseteq \{b_1,b_2,\dots,b_{t-1}\}\).  Since \(T\in \calC(G)\), then \(\varphi(\frakP_T(K_2,G))=\frakP_T(K_2,P)\).  Notice that if \(T\in \calC(G)\setminus\calC(P)\), then \(\frakP_T(K_2,P)\) is a non-minimal prime ideal containing \(\calJ_{K_2,P}\).

    Now let \(h\) be a homogeneous polynomial in \(S\) such that \((\calJ_{K_2, G}^k:h)=\frakP_T(K_2,G)\).  By \Cref{lem:change_initial}, we may assume that \(\ini(h)\notin \ini(\calJ_{K_2, G}^k)\).  Since \(\ini(h)\ini(\frakP_T(K_2,G))\subseteq \ini(\calJ_{K_2, G}^k)\), we obtain
    \[
        \varphi(\ini(h))\varphi(\ini(\frakP_T(K_2,G)))\subseteq \varphi(\ini(\calJ_{K_2, G}^k)),
    \]
    that is,
    \[
        \varphi(\ini(h))\ini(\frakP_T(K_2,P))\subseteq \ini(\calJ_{K_2, P}^k).
    \]
    Let \(\frakP_{T'}(K_2,P)\) be a minimal prime of \(\calJ_{K_2, P}\) contained in
    \(\frakP_T(P)\).  Then
    \[
        \varphi(\ini(h))\ini(\frakP_{T'}(K_2,P))\subseteq \ini(\calJ_{K_2,P}^k).
    \]
    By the proof of \Cref{powers1}, this implies that
    \begin{align*}
        \deg(h)
        &= \deg(\ini(h)) = \deg(\varphi(\ini(h))) 
        \stackrel{\text{\eqref{eqn:deg_h}}}{\ge} \v_{T'}(\calJ_{K_2, P})+2(k-1) \\
        &\ge \v(\calJ_{K_2,P})+2(k-1) = \v(\calJ_{K_2, G})+2(k-1),
    \end{align*}
    where the last equality holds because the \(\v\)-numbers of \(\calJ_{K_2, P}\) and \(\calJ_{K_2, G}\) depend only on the number of their maximal cliques by \Cref{CMclsoedgraph}.
\end{proof}

The following observation clarifies the behavior of local \(\v\)-numbers under the reduction to the spine subgraph.

\begin{Remark}
    If \(T \in \calC(P) \subseteq \calC(G)\), then \(T' = T\) holds in the proof of \Cref{powers2}.
    Consequently, for every positive integer \(k\),
    \[
        \v_T(\calJ_{K_2,G}^k) = \v_T(\calJ_{K_2,G})+2(k - 1).
    \]
\end{Remark}

As an immediate consequence, we obtain an explicit formula for path graphs, which are the simplest examples of Cohen--Macaulay closed graphs.

\begin{Corollary}
    If \(G=P_n\) is a path graph with  $n$ vertices, then for any integer $k\ge 1$,
    \[
        \v (\calJ_{K_2, G}^k) = \left\lceil\frac{2(n-2)}{3}\right\rceil+2(k-1).
    \]
\end{Corollary}
\begin{proof}
    Note that \(G\) is a connected Cohen--Macaulay closed graph with precisely  \(n-1\) maximal cliques, $\v (\calJ_{K_2, G}) = \left\lceil\frac{2(n-2)}{3}\right\rceil$ by Corollary \ref{CMclsoedgraph}. The result follows from  Proposition  \ref{powers2}.
\end{proof}

We conclude this section with a remark that indicates possible extensions of our methods to complete graphs \(K_m\) with \(m\ge 3\).

\begin{Remark}
    Let \(G\) be a path graph and let \(T\) be a cut set of \(G\).
    \begin{enumerate}[a]
        \item For any integer \(m \ge 2\), the method used in the proof of \Cref{lem:power_colon_g} can be adapted to show that
            \[
                \v_{T}(\calJ_{K_m,G}^k) \le \v_{T}(\calJ_{K_m,G}) + 2(k-1).
            \]
            However, equality need not hold in general. For instance, if \(G = P_5\) and \(T = \{3\}\), a direct computation yields
            \[
                \v(\calJ_{K_3,G}) = \v_{T}(\calJ_{K_3,G}) = 2
            \]
            and
            \[
                \v(\calJ_{K_3,G}^2) = \v_{T}(\calJ_{K_3,G}^2) = 3.
            \]
        \item Nevertheless, the modules
            \[
                (\calJ_{K_m,G}^k:\frakP_T(K_m,G)) \big/ \calJ_{K_m,G}^k
            \]
            and
            \[
                \bigl((\ini(\calJ_{K_m,G}))^k:\ini(\frakP_T(K_m,G))\bigr) \big/ (\ini(\calJ_{K_m,G}))^k
            \]
            seem to be closely related. It would be worthwhile to investigate further consequences of this relation.
    \end{enumerate}
\end{Remark}

\begin{acknowledgment*}
    The authors are grateful to the software system \texttt{Macaulay2} \cite{M2}, for serving as excellent sources of inspiration.  This work is supported by the Natural Science Foundation of Jiangsu Province (No.~BK20221353). In addition, the first author acknowledges partial support from ``the Fundamental Research Funds for Central Universities'' and ``Quantum Science and Technology - National Science and Technology Major Project'' (2021ZD0302902). And the second author is  partially supported by  the National Natural Science Foundation of China (No. 12471246).
\end{acknowledgment*}

\medskip
\hspace{-6mm} {\bf Data availability statement}

\vspace{3mm}
\hspace{-6mm}  The data used to support the findings of this study are included within the article.

\medskip
\hspace{-6mm} {\bf Conflict of interest}

\vspace{3mm}
\hspace{-6mm}  The authors declare that they have no competing interests.

\bibliography{references}

@Article{MR4792768,
  author   = {Jaramillo-Velez, Delio and Seccia, Lisa},
  journal  = {Collect. Math.},
  title    = {Connected domination in graphs and $\textup{v}$-numbers of binomial edge ideals},
  year     = {2024},
  issn     = {0010-0757,2038-4815},
  pages    = {771--793},
  volume   = {75},
  doi      = {10.1007/s13348-023-00412-w},
  fjournal = {Collectanea Mathematica},
  mrclass  = {05E40 (05C69 13A70 13P10)},
  mrnumber = {4792768},
  url      = {https://doi.org/10.1007/s13348-023-00412-w},
}

@article {MR3290687,
    AUTHOR = {Ene, Viviana and Herzog, J\"{u}rgen and Hibi, Takayuki and Qureshi, Ayesha Asloob},
     TITLE = {The binomial edge ideal of a pair of graphs},
   JOURNAL = {Nagoya Math. J.},
  FJOURNAL = {Nagoya Mathematical Journal},
    VOLUME = {213},
      YEAR = {2014},
     PAGES = {105--125},
      ISSN = {0027-7630},
   MRCLASS = {13F20 (05C25 13P10)},
  MRNUMBER = {3290687},
MRREVIEWER = {Monica La Barbiera},
       DOI = {10.1215/00277630-2389872},
       URL = {https://doi.org/10.1215/00277630-2389872},
}

@book {MR1251956,
    AUTHOR = {Bruns, Winfried and Herzog, J\"urgen},
     TITLE = {Cohen-{M}acaulay rings},
    SERIES = {Cambridge Studies in Advanced Mathematics},
    VOLUME = {39},
 PUBLISHER = {Cambridge University Press, Cambridge},
      YEAR = {1993},
     PAGES = {xii+403},
      ISBN = {0-521-41068-1},
   MRCLASS = {13H10 (13-02)},
  MRNUMBER = {1251956},
MRREVIEWER = {Matthew\ Miller},
}

@article {MR2669070,
    AUTHOR = {Herzog, J\"{u}rgen and Hibi, Takayuki and Hreinsd\'{o}ttir, Freyja and Kahle, Thomas and Rauh, Johannes},
     TITLE = {Binomial edge ideals and conditional independence statements},
   JOURNAL = {Adv. in Appl. Math.},
  FJOURNAL = {Advances in Applied Mathematics},
    VOLUME = {45},
      YEAR = {2010},
     PAGES = {317--333},
      ISSN = {0196-8858},
   MRCLASS = {13P10 (05C25 13F20)},
  MRNUMBER = {2669070},
MRREVIEWER = {Seth Sullivant},
       DOI = {10.1016/j.aam.2010.01.003},
       URL = {https://doi.org/10.1016/j.aam.2010.01.003},
}

@article {MR2782571,
    AUTHOR = {Ohtani, Masahiro},
     TITLE = {Graphs and ideals generated by some 2-minors},
   JOURNAL = {Comm. Algebra},
  FJOURNAL = {Communications in Algebra},
    VOLUME = {39},
      YEAR = {2011},
     PAGES = {905--917},
      ISSN = {0092-7872},
   MRCLASS = {13C40 (13P10)},
  MRNUMBER = {2782571},
MRREVIEWER = {Marcel Morales},
       DOI = {10.1080/00927870903527584},
       URL = {https://doi.org/10.1080/00927870903527584},
}

@article {MR3011436,
    AUTHOR = {Rauh, Johannes},
     TITLE = {Generalized binomial edge ideals},
   JOURNAL = {Adv. in Appl. Math.},
  FJOURNAL = {Advances in Applied Mathematics},
    VOLUME = {50},
      YEAR = {2013},
     PAGES = {409--414},
      ISSN = {0196-8858},
   MRCLASS = {13F20 (05C25 13P10)},
  MRNUMBER = {3011436},
MRREVIEWER = {Siamak Yassemi},
       DOI = {10.1016/j.aam.2012.08.009},
       URL = {https://doi.org/10.1016/j.aam.2012.08.009},
}

@Article{MR4869330,
  author   = {Dey, Deblina and Jayanthan, A. V. and Saha, Kamalesh},
  journal  = {Internat. J. Algebra Comput.},
  title    = {On the $\textup{v}$-number of binomial edge ideals of some classes of graphs},
  year     = {2025},
  issn     = {0218-1967,1793-6500},
  pages    = {119--143},
  volume   = {35},
  doi      = {10.1142/S0218196724500607},
  fjournal = {International Journal of Algebra and Computation},
  mrclass  = {13F65 (05C69 05E40 13F55)},
  mrnumber = {4869330},
  url      = {https://doi.org/10.1142/S0218196724500607},
}

@Article{MR4834446,
  author     = {Ambhore, Siddhi Balu and Saha, Kamalesh and Sengupta, Indranath},
  journal    = {Acta Math. Vietnam.},
  title      = {The $\textup{v}$-number of binomial edge ideals},
  year       = {2024},
  issn       = {0251-4184,2315-4144},
  pages      = {611--628},
  volume     = {49},
  doi        = {10.1007/s40306-024-00540-w},
  fjournal   = {Acta Mathematica Vietnamica},
  mrclass    = {13F20 (05E40 13F65)},
  mrnumber   = {4834446},
  mrreviewer = {Aryampilly\ V.\ Jayanthan},
  url        = {https://doi.org/10.1007/s40306-024-00540-w},
}

@Article{MR2863365,
  author     = {Ene, Viviana and Herzog, J\"urgen and Hibi, Takayuki},
  journal    = {Nagoya Math. J.},
  title      = {Cohen--{M}acaulay binomial edge ideals},
  year       = {2011},
  issn       = {0027-7630,2152-6842},
  pages      = {57--68},
  volume     = {204},
  doi        = {10.1215/00277630-1431831},
  fjournal   = {Nagoya Mathematical Journal},
  mrclass    = {13F20 (05C25 05E40 13C14 13D02)},
  mrnumber   = {2863365},
  mrreviewer = {Adam\ L.\ Van Tuyl},
  url        = {https://doi.org/10.1215/00277630-1431831},
}

@Article{MR4423525,
  author     = {Bolognini, Davide and Macchia, Antonio and Strazzanti, Francesco},
  journal    = {J. Algebraic Combin.},
  title      = {Cohen--{M}acaulay binomial edge ideals and accessible graphs},
  year       = {2022},
  issn       = {0925-9899,1572-9192},
  pages      = {1139--1170},
  volume     = {55},
  doi        = {10.1007/s10801-021-01088-w},
  fjournal   = {Journal of Algebraic Combinatorics. An International Journal},
  mrclass    = {13H10 (05E40 13C05)},
  mrnumber   = {4423525},
  url        = {https://doi.org/10.1007/s10801-021-01088-w},
}

@InCollection{MR4143239,
  author     = {Ene, Viviana and Herzog, J\"urgen},
  booktitle  = {Combinatorial structures in algebra and geometry},
  publisher  = {Springer, Cham},
  title      = {On the symbolic powers of binomial edge ideals},
  year       = {2020},
  isbn       = {978-3-030-52111-0; 978-3-030-52110-3},
  pages      = {43--50},
  series     = {Springer Proc. Math. Stat.},
  volume     = {331},
  doi        = {10.1007/978-3-030-52111-0\_4},
  mrclass    = {13P10 (05E40 13C15)},
  mrnumber   = {4143239},
  mrreviewer = {Nguyen Cong Minh},
  url        = {https://doi.org/10.1007/978-3-030-52111-0_4},
}

@article {MR4011111,
    AUTHOR = {Cooper, Susan M. and Seceleanu, Alexandra and Toh\u{a}neanu, \c{S}tefan O. and Pinto, Maria Vaz and Villarreal, Rafael H.},
     TITLE = {Generalized minimum distance functions and algebraic
              invariants of {G}eramita ideals},
   JOURNAL = {Adv. in Appl. Math.},
  FJOURNAL = {Advances in Applied Mathematics},
    VOLUME = {112},
      YEAR = {2020},
     PAGES = {101940, 34},
      ISSN = {0196-8858,1090-2074},
   MRCLASS = {13P25 (11T71 13C40 14G50 94B27)},
  MRNUMBER = {4011111},
MRREVIEWER = {C\'icero\ Carvalho},
       DOI = {10.1016/j.aam.2019.101940},
       URL = {https://doi.org/10.1016/j.aam.2019.101940},
}

@article {MR3169597,
    AUTHOR = {Mohammadi, Fatemeh and Sharifan, Leila},
     TITLE = {Hilbert function of binomial edge ideals},
   JOURNAL = {Comm. Algebra},
  FJOURNAL = {Communications in Algebra},
    VOLUME = {42},
      YEAR = {2014},
     PAGES = {688--703},
      ISSN = {0092-7872,1532-4125},
   MRCLASS = {13F20 (05E40 13C15 13D40)},
  MRNUMBER = {3169597},
MRREVIEWER = {Lucia\ Maria\ Marino},
       DOI = {10.1080/00927872.2012.721037},
       URL = {https://doi.org/10.1080/00927872.2012.721037},
}

@article {MR4984037,
    AUTHOR = {Kumar, Arvind and Pomeroy, Joshua and Tran, Le},
     TITLE = {Binomial edge ideals of crown graphs},
   JOURNAL = {J. Algebraic Combin.},
  FJOURNAL = {Journal of Algebraic Combinatorics. An International Journal},
    VOLUME = {62},
      YEAR = {2025},
     PAGES = {51},
      ISSN = {0925-9899,1572-9192},
   MRCLASS = {13C10 (05E40 13C15 13F20)},
  MRNUMBER = {4984037},
       DOI = {10.1007/s10801-025-01464-w},
       URL = {https://doi.org/10.1007/s10801-025-01464-w},
}

@article {MR4741232,
    AUTHOR = {Conca, Aldo},
     TITLE = {A note on the $\textup{v}$-invariant},
   JOURNAL = {Proc. Amer. Math. Soc.},
  FJOURNAL = {Proceedings of the American Mathematical Society},
    VOLUME = {152},
      YEAR = {2024},
     PAGES = {2349--2351},
      ISSN = {0002-9939,1088-6826},
   MRCLASS = {13A30},
  MRNUMBER = {4741232},
MRREVIEWER = {Tony\ J.\ Puthenpurakal},
       DOI = {10.1090/proc/16767},
       URL = {https://doi.org/10.1090/proc/16767},
}

@article {arXiv:2306.14243,
    AUTHOR = {Ficarra, Antonino and Sgroi, Emanuele},
     TITLE = {Asymptotic behaviour of the $\textup{v}$-number of homogeneous ideals},
   eprint = {arXiv:2306.14243},
  MRCLASS = {13E05},
       URL = {https://arxiv.org/pdf/2306.14243},
}

@article {arXiv:2507.02161,
    AUTHOR = {Emiliano Liwski},
     TITLE = {The $\V$-Number of Binomial Edge Ideals: Minimal Cuts and Cycle Graphs},
   eprint = {arXiv:2507.02161},
  year = {2025},
       URL = {https://arxiv.org/pdf/2507.02161},
}

@article {MR4139109,
    AUTHOR = {Jaramillo, Delio and Villarreal, Rafael H.},
     TITLE = {The $\textup{v}$-number of edge ideals},
   JOURNAL = {J. Combin. Theory Ser. A},
  FJOURNAL = {Journal of Combinatorial Theory. Series A},
    VOLUME = {177},
      YEAR = {2021},
     PAGES = {Paper No. 105310, 35},
      ISSN = {0097-3165,1096-0899},
   MRCLASS = {05E40 (05C65 13A15)},
  MRNUMBER = {4139109},
MRREVIEWER = {Margherita\ Barile},
       DOI = {10.1016/j.jcta.2020.105310},
       URL = {https://doi.org/10.1016/j.jcta.2020.105310},
}

@Article{MR4942694,
  author   = {Kataoka, Tatsuya and Muta, Yuji and Terai, Naoki},
  journal  = {J. Algebra},
  title    = {The $\textup{v}$-numbers of {S}tanley--{R}eisner ideals from the viewpoint of {A}lexander dual complexes},
  year     = {2025},
  issn     = {0021-8693,1090-266X},
  pages    = {589--611},
  volume   = {684},
  doi      = {10.1016/j.jalgebra.2025.07.017},
  fjournal = {Journal of Algebra},
  mrclass  = {13F55 (13H10)},
  mrnumber = {4942694},
  url      = {https://doi.org/10.1016/j.jalgebra.2025.07.017},
}

@Article{MR4756096,
  author     = {Saha, Kamalesh},
  journal    = {Int. Math. Res. Not. IMRN},
  title      = {The $\textup{v}$-number and {C}astelnuovo--{M}umford regularity of cover ideals of graphs},
  year       = {2024},
  issn       = {1073-7928,1687-0247},
  pages      = {9010--9019},
  doi        = {10.1093/imrn/rnad277},
  fjournal   = {International Mathematics Research Notices. IMRN},
  mrclass    = {13A70 (05E40)},
  mrnumber   = {4756096},
  mrreviewer = {Huy\ T\`ai\ H\`a},
  url        = {https://doi.org/10.1093/imrn/rnad277},
}

@Article{MR4447411,
  author     = {Ene, Viviana and Rinaldo, Giancarlo and Terai, Naoki},
  journal    = {Res. Math. Sci.},
  title      = {Sequentially {C}ohen--{M}acaulay binomial edge ideals of closed graphs},
  year       = {2022},
  issn       = {2522-0144,2197-9847},
  pages      = {Paper No. 39, 17},
  volume     = {9},
  doi        = {10.1007/s40687-022-00334-2},
  fjournal   = {Research in the Mathematical Sciences},
  mrclass    = {13H10 (05E40 13C70)},
  mrnumber   = {4447411},
  mrreviewer = {Jorge\ Neves},
  url        = {https://doi.org/10.1007/s40687-022-00334-2},
}

@article {MR1711319,
    AUTHOR = {Cutkosky, S. Dale and Herzog, J{\"u}rgen and Ng{\^o} Vi{\^e}t Trung},
     TITLE = {Asymptotic behaviour of the {C}astelnuovo-{M}umford
              regularity},
   JOURNAL = {Compositio Math.},
  FJOURNAL = {Compositio Mathematica},
    VOLUME = {118},
      YEAR = {1999},
     PAGES = {243--261},
      ISSN = {0010-437X,1570-5846},
   MRCLASS = {13D45 (13A30 13C99 14F17)},
  MRNUMBER = {1711319},
MRREVIEWER = {Vincenzo\ Di Gennaro},
       DOI = {10.1023/A:1001559912258},
       URL = {https://doi.org/10.1023/A:1001559912258},
}

@Article{MR1621961,
  author     = {Kodiyalam, Vijay},
  journal    = {Proc. Amer. Math. Soc.},
  title      = {Asymptotic behaviour of {C}astelnuovo--{M}umford regularity},
  year       = {2000},
  issn       = {0002-9939,1088-6826},
  pages      = {407--411},
  volume     = {128},
  doi        = {10.1090/S0002-9939-99-05020-0},
  fjournal   = {Proceedings of the American Mathematical Society},
  mrclass    = {13D45 (14B15)},
  mrnumber   = {1621961},
  mrreviewer = {P.\ Schenzel},
  url        = {https://doi.org/10.1090/S0002-9939-99-05020-0},
}

@article {MR4491066,
    AUTHOR = {Saha, Kamalesh and Sengupta, Indranath},
     TITLE = {The $\textup{v}$-number of monomial ideals},
   JOURNAL = {J. Algebraic Combin.},
  FJOURNAL = {Journal of Algebraic Combinatorics. An International Journal},
    VOLUME = {56},
      YEAR = {2022},
     PAGES = {903--927},
      ISSN = {0925-9899,1572-9192},
   MRCLASS = {13F55 (05C70 05E40 13A15 13A70)},
  MRNUMBER = {4491066},
MRREVIEWER = {Somayeh\ Bandari},
       DOI = {10.1007/s10801-022-01137-y},
       URL = {https://doi.org/10.1007/s10801-022-01137-y},
}

@article {MR4950319,
    AUTHOR = {Biswas, Prativa and Mandal, Mousumi},
     TITLE = {A study of $\textup{v}$-number for some monomial ideals},
   JOURNAL = {Collect. Math.},
  FJOURNAL = {Collectanea Mathematica},
    VOLUME = {76},
      YEAR = {2025},
     PAGES = {667--682},
      ISSN = {0010-0757,2038-4815},
   MRCLASS = {05E40 (05C38 05C69 13F20 13F55)},
  MRNUMBER = {4950319},
       DOI = {10.1007/s13348-024-00451-x},
       URL = {https://doi.org/10.1007/s13348-024-00451-x},
}

@Misc{M2,
    author = {Grayson, Daniel R. and Stillman, Michael E.},
    title = {Macaulay2, a software system for research in algebraic geometry},
    howpublished = {Available at \url{https://math.uiuc.edu/Macaulay2/}},
    eprint = {Available at \url{https://math.uiuc.edu/Macaulay2/}},
    note = {Available at \url{https://math.uiuc.edu/Macaulay2/}},
}

@article {MR4932665,
    AUTHOR = {Ficarra, Antonino and Macias Marques, Pedro},
     TITLE = {The {$\rm v$}-function of powers of sums of ideals},
   JOURNAL = {J. Algebraic Combin.},
  FJOURNAL = {Journal of Algebraic Combinatorics. An International Journal},
    VOLUME = {62},
      YEAR = {2025},
     PAGES = {Paper No. 13, 21},
      ISSN = {0925-9899,1572-9192},
   MRCLASS = {13F20 (05C70 05E40 13F55)},
  MRNUMBER = {4932665},
}

@article {MR4950310,
    AUTHOR = {Ficarra, Antonino},
     TITLE = {Simon conjecture and the {$\rm v$}-number of monomial
              ideals},
   JOURNAL = {Collect. Math.},
  FJOURNAL = {Collectanea Mathematica},
    VOLUME = {76},
      YEAR = {2025},
     PAGES = {477--492},
      ISSN = {0010-0757,2038-4815},
   MRCLASS = {13F20 (05C70 05E40 13F55)},
  MRNUMBER = {4950310},
}
\end{document}